\documentclass[12pt,reqno]{amsart}
\usepackage{a4wide}
\usepackage{amsfonts,amsmath,amssymb,amscd}
\usepackage{amsthm}
\usepackage{amsmath,todonotes}
\usepackage{amscd}
\usepackage{cancel}
\usepackage{verbatim}
\usepackage{color}
\usepackage[all]{xy}
\usepackage[mathscr]{eucal}
\usepackage{mathrsfs}
\usepackage{fullpage}
\usepackage{enumitem}
\usepackage{hyperref}
\usepackage{bbm}
\usepackage{qtree}
\usepackage{mathrsfs}
\usepackage{bm}
\usepackage{bigints}
\usepackage{url}
\allowdisplaybreaks
\let\dotlessi=\i
\newcommand{\pFq}[5]{\ensuremath{{}_{#1}F_{#2} \left( \genfrac{}{}{0pt}{}{#3}
		{#4} \bigg| {#5} \right)}}

\renewcommand{\l}{\lambda}

\usepackage[headheight=110pt,top=0.6in, bottom=0.7in, left=0.6in, right=0.6in]{geometry}
\numberwithin{equation}{section}

\newtheorem{remark}[subsection]{Remark}

\newtheorem{theorem}{Theorem}

\newtheorem{lemma}[theorem]{Lemma}

\newtheorem{corollary}[theorem]{Corollary}

\theoremstyle{definition}

\numberwithin{theorem}{section} 
\numberwithin{equation}{section}
\numberwithin{table}{section}

\newcommand{\re}{\textnormal{Re}}

\renewcommand{\[}{\left[}

\makeatletter
\def\proof{\@ifnextchar[{\@oproof}{\@nproof}}
\def\@oproof[#1][#2]{\trivlist\item[\hskip\labelsep\textit{#2 Proof of\
		#1.}~]\ignorespaces}
\def\@nproof{\trivlist\item[\hskip\labelsep\textit{Proof.}~]\ignorespaces}
\begin{document}
	\title[]{A divisor function of Wigert and higher degree forms}
	\author{Debika Banerjee, Atul Dixit and Rajat Gupta}
	%	\author{Atul Dixit}
	%	\author{Shivajee Gupta}
	\address{Department of Mathematics, Indraprastha Institute of Information Technology Delhi, Okhla Phase III, New Delhi 110020, India}
	\email{debika@iiitd.ac.in}
	\address{Department of Mathematics, Indian Institute of Technology Gandhinagar, Palaj, Gandhinagar, Gujarat 382355, India}
	\email{adixit@iitgn.ac.in}
	\address{Department of Mathematics, Indian Institute of Technology Jodhpur, Karwar, Jodhpur,  Rajasthan 342030, India}
	\email{rgupta@iitj.ac.in}
	\begin{abstract}
	Let $k\in\mathbb{N}$. Wigert's divisor function $d^{\left(\frac{1}{k}\right)}(j)$ counts the number of representations of $j$ of the form $m^k+mn$ with $m\geq1	, n\geq0$. Let $\mathcal{F}_k(s)$ denote the Dirichlet series of $d^{\left(\frac{1}{k}\right)}(j)$. While  $\mathcal{F}_2(s)$  is essentially a well-known special case of the Euler-Zagier double zeta function, and hence well-studied, very little is known about  $\mathcal{F}_k(s)$  for $k>2$. We offer three new representations for $\mathcal{F}_k(s)$ for $k\geq2$, one of which is an analogue of the Chowla-Selberg formula as well as of a formula of Atkinson. The meromorphicity of  $\mathcal{F}_k(s)$  is also discussed. The special value  $\mathcal{F}_3\left(\frac{3}{2}\right)$  is expressed in terms of an infinite series of Bessel functions and a generalized divisor function.
	\end{abstract}
	\maketitle
	\tableofcontents
	
	\section{Introduction}\label{intro}
	Let $k\in\mathbb{N}$. Wigert  \cite{wig} introduced the divisor function
	\begin{equation*}
		d^{\left(\frac{1}{k}\right)}(n):=\sum_{d|n\atop{ d\leq n^{1/k}}}1.
	\end{equation*}	
	Clearly, $d^{\left(\frac{1}{1}\right)}(n)=d(n)$, the number of divisors of $n$. For Re$(y)>0$, he showed that 
	\begin{equation}\label{d1k-lambert}
		\sum_{n=1}^{\infty}d^{\left(\frac{1}{k}\right)}(n)e^{-ny}=\sum_{n=1}^{\infty}\frac{e^{-n^{k}y}}{1-e^{-ny}}.
	\end{equation}
	and obtained
	\begin{equation*}
		\sum_{n\leq x}d^{(\frac{1}{k})}(n)=\frac{x}{k}(\log(x)+k\gamma-1)+O\left(x^{1-\frac{1}{k}}\right),
	\end{equation*}
	where $\gamma$ is Euler's constant.
	%Equivalently, letting $q=e^{-z}$ so that for $|q|<1$, we have 
	%\begin{equation}
	%\sum_{n=1}^{\infty}d^{\left(\frac{1}{N}\right)}(n)q^n=\sum_{n=1}^{\infty}\frac{q^{n^N}}{1-q^n}.
	%\end{equation}
	
	After deducing \eqref{d1k-lambert}, Wigert writes, \textit{`Cependent la s\'{e}rie \dots n'est pas d'un facile acc\`{e}s quand on veut
		\'{e}tudier les propri\'{e}t\'{e}s de la fonction qu'elle repr\'{e}sente. Je pr\'{e}f\`{e}re donc de m'occuper d'une autre généralisation, semblable \`{a} la préc\'{e}dente.'}, which, in English, translates to \textit{`However the series \dots is not easy to work with when you want to study the properties of the function it represents. I therefore prefer to deal with another generalization, similar to the previous one.'} 
	
	In the same paper \cite{wig}, Wigert then studies a different arithmetic function, namely,
	\begin{equation*}
		d^{(k)}(n):=\sum_{d^k|n}1,	
	\end{equation*} 
	for a fixed $k\in\mathbb{N}$. 
This function as well as the Lambert series associated to it, and their generalizations, have received a lot of attention in recent years; see \cite{banerjee-maji}, \cite{badigu}, \cite{DMV}, \cite{guptamaji},  \cite{koshwigleningrad}, \cite{robles-roy} and \cite{wigannalen}. For example, a generalization of  $d^{(k)}(n)$, namely,
\begin{equation}\label{sigmakz(n)} \sigma_{z}^{(k)}(n):=\sum_{d^k|n}d^z\hspace{10mm}(k\in\mathbb{N}, z\in\mathbb{C}),
\end{equation}
was recently studied in \cite{DMV} and a  Vorono\"{\dotlessi} summation formula was obtained for it. This function will also turn up in the results of the	 current paper.

However, the literature on the divisor function Wigert initially set out to work in \cite{wig}, that is, on $d^{\left(\frac{1}{k}\right)}(n)$, is quite scarce.  Bordell\`{e}s \cite{bordelles} obtained an asymptotic formula for the average order of $d^{\left(\frac{1}{k}\right)}(n)$ which was improved by Banerjee and Khyati \cite{debikakhyati}. There seem to be no other studies on $d^{\left(\frac{1}{k}\right)}(n)$ from the point of view of analytic number theory besides these two. 

For $a\in\mathbb{C}$, Bordell\`{e}s  \cite{bordelles} also considered the generalized divisor function
\begin{equation*}
	\sigma_{a}^{\left(\frac{1}{k}\right)}(n):=\sum_{d|n\atop d\leq n^{1/k}}d^{a}.
\end{equation*}
It is easy to see that for Re$(y)>0$,
\begin{equation*}
	\mathcal{J}_{a,k}(y):=	\sum_{n=1}^{\infty}\sigma_{a}^{\left(\frac{1}{k}\right)}(n)e^{-ny}=\sum_{n=1}^{\infty}\frac{n^ae^{-n^{k}y}}{1-e^{ny}}.
\end{equation*}
Equivalently, letting $q=e^{-y}$ so that  $|q|<1$, we have 
\begin{equation*}
	\mathscr{J}_{a,k}(q):=\mathcal{J}_{a-1,k}(y):=\sum_{n=1}^{\infty}\frac{n^{a-1}q^{n^k}}{1-q^n}.
\end{equation*}
It is trivial to observe that
\begin{align*}
	\mathscr{J}_{1,k}(n)=\sum_{n=1}^{\infty}d^{\left(\frac{1}{k}\right)}(n)q^n,\hspace{4mm}	\mathscr{J}_{1,1}(n)=\sum_{n=1}^{\infty}d(n)q^n,
\end{align*}
and that for an even integer $j>1$, 	$\mathscr{J}_{j,1}(n)$ is essentially what occurs in the Fourier expansion of the Eisenstein series $E_j$ on $\textup{SL}_{2}\left(\mathbb{Z}\right)$.

For $j\in\mathbb{Z}$, the function $\mathscr{J}_{j,2}(q)$ are generating functions of Joyce invariants which appeared in a work of Mellit and Okada \cite{mellitokada} on the moduli stacks of semistable objects with Donaldson–Thomas type invariants, introduced by Joyce, for stability conditions on $K3$ surfaces. Mellit and Okada showed the invariance of moduli stacks on faithful stability conditions and motivic invariants. 
They also showed \cite[Theorem 1.2]{mellitokada} that the function
\begin{equation*}
	-\frac{1}{24}-\frac{1}{2}\sum_{n=1}^{\infty}nq^{n^2}+\mathscr{J}_{2, 2}(q)=	-\frac{1}{24}+\frac{1}{2}\sum_{n=-\infty\atop n\neq 0}^{\infty}\frac{nq^{n^2}}{1-q^n}
\end{equation*}
is a mock modular form on $\textup{SL}_{2}\left(\mathbb{Z}\right)$. Bringmann \cite[Theorem 1.3]{bringmann2018} worked with $\mathscr{J}_{k,2}(q)$ for even $k\geq 2$, and generalized the result of Mellit and Okada.

As is the case with many arithmetic functions related to the divisors of integers,  there are two main streams of research involving $d^{\left(\frac{1}{k}\right)}(n)$: one originates from analytic number theory and concerns the analytic properties of the Dirichlet series of $d^{\left(\frac{1}{k}\right)}(n)$, its summatory function and the analyses of its error term  \cite{debikakhyati}, \cite{bordelles}, while the other has its genesis in the theory of modular and mock modular forms \cite{bringmann2018},  \cite{mellitokada}. 

In this first among a series of papers, we focus on the first aspect. Our objective is to study the Dirichlet series 	$\mathcal{F}_k(s)$ of $d^{\left(\frac{1}{k}\right)}(n)$, in particular, obtain its meromorphic continuation in the entire complex plane and derive a Chowla-Selberg-type formula for it. See Theorems \ref{ac} and \ref{cs type formula}.

Observe that for Re$(s)>1$ and $k\in\mathbb{N}$,
\begin{align}\label{f2i}
	\mathcal{F}_k(s):=\sum_{j=1}^{\infty} \frac{d^{\left(\frac{1}{k}\right)}(j)}{j^s}
	=\sum_{j=1}^{\infty}\frac{1}{j^s}\sum_{m|j\atop m\leq j^{1/k}}1
	=\sum_{\substack{m=1\\ n=0}}^{\infty}\frac{1}{(m^k+mn)^s},
	%=\mathcal{L}_k(s)+\zeta(ks).
\end{align}
which can be written in the form
\begin{align}\label{f2}
	\mathcal{F}_k(s)=\zeta(ks)+\mathcal{L}_k(s),
\end{align}
where $\zeta(s)$ is the Riemann zeta function defined for Re$(s)>1$ by
\begin{equation*}
	\zeta(s)=\sum_{n=1}^{\infty}\frac{1}{n^{s}},
\end{equation*}
and
\begin{equation}\label{l2i}
	\mathcal{L}_k(s):=\sum_{\substack{m, n=1}}^{\infty}\frac{1}{(m^k+mn)^s}.
\end{equation}
From \eqref{f2i}, one sees that $d^{\left(\frac{1}{k}\right)}(j)$ is nothing but the number of representations of $j$ in forms of degree $k$ of the form $m^k+mn$, all of which are inhomogeneous except when $n=0$ and/or $k=2$.

The special case $k=2$ of \eqref{l2i} has been well-studied, and is a specialization of the Euler-Zagier double zeta function defined by \cite{atkinson}, \cite{kiuchitanigawa}
\begin{equation*}
	\zeta_2(s_1, s_2):=\sum_{1\leq n_1<n_2}\frac{1}{n_1^{s_1}n_2^{s_2}}.
\end{equation*}
Indeed, 
the above series converges absolutely when Re$(s_1)>1$ and Re$(s_1+s_2)>2$. Atkinson \cite{atkinson} obtained the meromorphic continuation of $\zeta_2(s_1, s_2)$  in the whole space $\mathbb{C}^2$ and then used it to derive one of the central formulas in the analytic theory of the Riemann zeta function, now known as \emph{Atkinson's formula}. The literature on the double zeta function is vast and its  functional as well as approximate functional equations are well-known; see, for example, \cite{choiematsumoto}, \cite{kiuchitanigawa}. 

In \cite[Equations (1.3), (2.1)]{atkinson}, we find that for Re$(s)>1$,
\begin{equation}\label{by atkinson}
	\mathcal{L}_2(s)=\frac{1}{2}\left(\zeta^2(s)-\zeta(2s)\right).
\end{equation}
This is straightforward to derive upon noticing that
\begin{equation*}
	d^{\left(\frac{1}{2}\right)}(n)=\begin{cases}
		\frac{1}{2}d(n),\hspace{10mm}\text{if}\hspace{1mm}n\hspace{1mm}\text{is not a perfect square},\\
		\frac{1}{2}(d(n)+1),\text{if}\hspace{1mm}n\hspace{1mm}\text{is a perfect square}, 
	\end{cases}
\end{equation*}
since then
\begin{align*}
	\mathcal{F}_2(s)=\sum_{n=1}^{\infty}\frac{d^{\left(\frac{1}{2}\right)}(n)}{n^s}&=\sum_{n=1\atop n\hspace{0.5mm}\text{not a perfect square}}^{\infty}\frac{d(n)}{2n^s}	+\sum_{n=1\atop n\hspace{0.5mm}\text{ a perfect square}}^{\infty}\frac{(d(n)+1)}{2n^s}	\nonumber\\
	&=\frac{1}{2}\zeta^2(s)+\frac{1}{2}\zeta(2s),
\end{align*}
One can then appeal to the special case $k=2$ of  \eqref{f2}.

Kiuchi, Tanigawa and Zhai derived another representation for $\mathcal{F}_2(s)$  \cite[Theorem 1]{kiuchitanigawa} whose special case $s_1=s_2=s$ yields, for $0<\textup{Re}(s)<1$,
\begin{align}\label{f22spl}
	\mathcal{F}_2(s)=\frac{1}{2}\zeta(2s)+\frac{\zeta(2s-1)}{s-1}+ \frac{s}{\pi }  \sum_{n=1}^{\infty} \frac{\sigma_{1-2s}(n)}{n} \int_{1}^{\infty} \sin (2\pi n x )x^{-s-1	} dx.
\end{align}
Using the general result involving $s_1$ and $s_2$, they derived an approximate functional equation for the double zeta-function $\zeta(s_1,s_2)$.
They also showed \cite[Remark 1]{kiuchitanigawa} that the series occurring in this result is analogous to that occurring in a formula of Atkinson  \cite[Equations (2.3), (2.5)]{kiuchitanigawa}.

For $k>2$, however, the only known information about the Dirichlet series in \eqref{f2i} or the series in \eqref{l2i}, is that it can be meromorphically continued to the whole complex plane. This was done by Mellin \cite{mellin} (see also \cite[pp. 3-4]{komori-matsumoto-tsumura}) who actually worked with a more general Dirichlet series. But Mellin \cite{mellin} did not give the pole structure of his more general Dirichlet series.  Also, no transformations or explicit representations of $\mathcal{L}_k(s)$ or $\mathcal{F}_k(s)$ are known for $k>2$ which is what we accomplish in this paper. 

We give a new proof of the meromorphic continuation of $\mathcal{F}_k(s)$, which, additionally, gives the precise location of its poles. 

\begin{theorem}\label{ac}
	Let $k\in\mathbb{N}, k\geq2$.	The function $\mathcal{F}_k(s)$ can be meromorphically continued to the whole plane $\mathbb{C}$ with  a double pole at $s=1$, a simple pole at $s=\frac1k$,  and simple poles at $s=\frac1k-\frac{(k-1)(2n+1)}{k}$ for $n =0, 1, 2,\cdots$.
\end{theorem} 
Our second theorem gives a new representation for the Dirichlet series of $d^{\left(\frac{1}{k}\right)}(n)$ whose equivalent version (see \eqref{formula2} below) generalizes \eqref{f22spl}.
\begin{theorem}\label{ktz}
	Let $k\geq2$ and $\textup{Re}(s)>0$. Let the function $\sigma_z^{(k)}(n)$ be defined in \eqref{sigmakz(n)}. Then
	\begin{align}\label{formula3}
	\mathcal{F}_k(s) &= \frac{1}{2}\zeta(ks) + \frac{\zeta(ks - (k - 1))}{s - 1} + \frac{s}{12} \zeta(ks + k - 1) \nonumber\\
	&\quad- \frac{s(s + 1)}{2\pi^2} \sum_{n=1}^{\infty} \frac{\sigma_{k-1 - ks}^{(k-1)}(n)}{n^2} \int_1^{\infty} \cos(2\pi n x) x^{-s-2} \, dx.
\end{align}
\end{theorem}

While the series in \eqref{l2i} for $k=2$ has a simple representation in terms of the Riemann zeta function, namely, the one in \eqref{by atkinson}, for $k>2$, it involves the higher degree form $m^k+mn$, and, to the best of our knowledge, no representation for it in terms of well-studied functions is known. We obtain such a result in Theorem \ref{cs type formula}. Before discussing it, however, we recall the famous Chowla-Selberg formula.

	Let $Q(m,n) = am^2 + bmn + cn^2$ denote a positive definite quadratic form with $a,b$ and $c$ real. Then the Epstein zeta-function $Z(s, Q)$,  defined for Re$(s)>1$ by 
	\begin{equation*}
		Z(s, Q):=\sum_{m,n=-\infty\atop (m,n)\neq(0,0)}^{\infty}(Q(m,n))^{-s},
	\end{equation*}	
satisfies the following formula for $s\in\mathbb{C}$, namely,
\begin{align*}
	Z(s, Q)&:=2a^{-s}\zeta(2s)+2a^{-s}k^{1-2s}\pi^{\frac{1}{2}}\frac{\Gamma(s-\frac{1}{2})\zeta(2s-1)}{\Gamma(s)}\nonumber\\
	&\quad+\frac{8a^{-s}k^{\frac{1}{2}-s}\pi^s}{\Gamma(s)}\sum_{n=1}^{\infty}n^{s-\frac{1}{2}}\sigma_{1-2s}(n)K_{s-\frac{1}{2}}	(2\pi kn)\cos\left(\frac{n\pi b}{a}\right),
\end{align*}
where $k^2=|d|/(4a^2)$, $d=b^2-4ac$ and $K_{\nu}(z)$ is the modified Bessel function of the second kind defined in Section \ref{prelim}. It was first announced by Chowla and Selberg in \cite{chowla-selberg}. Our formula in Theorem \ref{cs type formula} may be conceived as a Chowla-Selberg-type formula for $\mathcal{L}_k(s)$, $k>2$. Next, we review a result of Atkinson. In \cite{atkinson}, he showed that if Re$(s_1)<0$, Re$(s_2)>1$ and Re$(s_1+s_2)>0$, then
\begin{align*}
\zeta_2(s_1,s_2)=\frac{\Gamma(s_1+s_2-1)\Gamma(1-s_1)}{\Gamma(s_2)}\zeta(s_1+s_2-1)+2\sum_{n=1}^{\infty}\sigma_{1-s_1-s_2}(n)\int_{0}^{\infty}y^{-s_1}(1+y)^{-s_2}\cos(2\pi ny)dy.
\end{align*}
Our result below is an analogue of the above formula of Atkinson as well.
\begin{theorem}\label{cs type formula}
	Let $k>2$ and $1<\textup{Re}(s)<k-1$. Let $\sigma_{z}^{(k)}(n)$ and  $\mathcal{F}_k(s)$ be defined in \eqref{sigmakz(n)} and \eqref{f2i} respectively. Then
	\begin{align*}
		\mathcal{F}_k(s)&=\zeta(ks)+\zeta^2(s)+\frac{1}{(k-1)\Gamma(s)}\Gamma\left(\frac{1-s}{k-1}\right)\Gamma\left(\frac{ks-1}{k-1}\right)\zeta\left(\frac{ks-1}{k-1}\right)\nonumber\\
		&\quad+\frac{2}{(k-1)\Gamma(s)}\sum_{n=1}^{\infty} \frac{\sigma^{(k-1)}_{ks-1}(n)}{n^{\frac{ks-1}{k-1}}}\int_{0}^{\infty}\left((1+x)^{-s}-1\right)x^{\frac{s+k-2}{1-k}}\cos\left(2\pi(nx)^{\frac{1}{k-1}}\right)\, dx.
	\end{align*}
\end{theorem}
The integral in the above theorem is quite tricky to evaluate. It is important to note that the standard version of Parseval's formula for Mellin transforms is inapplicable here. One has to resort to an extended form of Parseval's formula given in Theorem \ref{vu kim tuan}. Before we state  this integral evaluation, we introduce a compact notation for a Meijer $G$-function frequently arising in our results, that is, we define
\begin{align}\label{short notation} \mathcal{G}_k(s,y):=G_{2, \, \,\,\,\,\,2k}^{k+1, \, \, 2}\left( \begin{matrix}
		\left\{\frac{2k-3+s}{2k-2}, \frac{k-2+s}{2k-2}\right\},\left\{\right\}\\
		\left\{0,\frac{1}{k-1},\cdots,\frac{k-2}{k-1}, \frac{sk-1}{2k-2}, \frac{sk-1}{2k-2}+\frac{1}{2}\right\}, \left\{\frac{1}{2k-2},\frac{3}{2k-2},\cdots,\frac{2k-3}{2k-2}\right\}
	\end{matrix} \bigg | \left(\frac{y}{2k-2}\right)^{2k-2} \right).
\end{align}
\begin{theorem}\label{integral conjecture}
Let $k\in\mathbb{N}, k\geq2$, and $y>0$. Let $\mathcal{G}_k(s,y)$ be defined in \eqref{short notation}. If $k=2$ and  $1<\textup{Re}(s)<2$, or if $k>2$ and  $1<\textup{Re}(s)<3$,  then
\begin{small}\begin{align}\label{integral conjecture eqn}
\int_{0}^{\infty}\left((1+x)^{-s}-1\right)x^{\frac{s+k-2}{1-k}}\cos\left(yx^{\frac{1}{k-1}}\right)\, dx=(k-1)\bigg\{\tfrac{2^{s-2}}{\sqrt{\pi(k-1)}\Gamma(s)}\mathcal{G}_k(y,s)-y^{s-1}\Gamma(1-s)\sin\left(\frac{\pi s}{2}\right)\bigg\}.
\end{align}
\end{small}
\normalsize
\end{theorem} 
Once the evaluation of the integral in terms of the Meijer $G$-function is obtained, we can employ the known asymptotics of the Meijer $G$-function to analytically continue Theorem \ref{cs type formula} in a wider region which includes the strip $1<\textup{Re}(s)<k-1$.  This extension is stated below.
\begin{theorem}\label{analytically extended}
%With the same hypotheses as in Theorem \ref{cs type formula}, we have
Let $k\in\mathbb{N}, k>2$ and let $M$ be a positive integer. For any $s$ in the strip $-2M-1<\textup{Re}(s)<(2M+2)(k-1)$,
%with the exception of $s=1, 1/k,$ and $\frac1k-\frac{(k-1)(2n+1)}{k}$ for $0\leq n <\frac{Mk+1}{k-1}$, we have
\small\begin{align}\label{analytically extended eqn}
	\mathcal{F}_k(s)&=\zeta(ks)+\zeta^2(s)+\frac{1}{(k-1)\Gamma(s)}\Gamma\left(\frac{1-s}{k-1}\right)\Gamma\left(\frac{ks-1}{k-1}\right)\zeta\left(\frac{ks-1}{k-1}\right)+\frac{2^{s-1}}{\sqrt{\pi(k-1)}\Gamma^2(s)}\nonumber\\
	&\times\bigg[\sum_{n=1}^{\infty} \frac{\sigma^{(k-1)}_{ks-1}(n)}{n^{\frac{ks-1}{k-1}}}\bigg\{ \mathcal{G}_k\left(s,2\pi n^{\frac{1}{k-1}}\right)-\frac{\pi^{s+\frac{1}{2}}\sqrt{k-1}}{\cos\left(\frac{\pi s}{2}\right)}n^{\frac{s-1}{k-1}}-\sum_{\mu=1}^{M}A_{s,k}(\mu)n^{\frac{s-1}{k-1}-2\mu}-\sum_{\mu=0}^{M}B_{s,k}(\mu)n^{\frac{s-k}{k-1}-2\mu}\bigg\}\nonumber\\
	%&\quad\nonumber\\
	&+\sum_{\mu=1}^{M}A_{s,k}(\mu)\zeta(s+2\mu)\zeta(1-s+2\mu(k-1))+\sum_{\mu=0}^{M}B_{s,k}(\mu)\zeta(s+1+2\mu)\zeta(k-s+2\mu(k-1))\bigg].
\end{align}
\end{theorem}
\normalsize
%$\frac{1}{2\sqrt{\pi}} G_{2,6}^{4,2} \left( \pi^2 n \middle| \begin{matrix} 1/2, 1/2 \\ 1/4, 0, -3/4, 0, -1/4, 3/4 \end{matrix} \right)$
Note that one cannot let $M=0$ in the above theorem since it does not give the correct strip in the $s$-complex plane where the result holds. However, this can be easily rectified by combining Theorems \ref{cs type formula} and \ref{integral conjecture}. The following identity we get after doing this is actually valid for\footnote{This is easily seen by putting $M=0$ in \eqref{meijer G asymptotic}.} $0<\textup{Re}(s)<k-1$:
	\begin{align}\label{lks final other}
	\mathcal{F}_k(s)&=\zeta(ks)+\zeta^2(s)+\frac{1}{(k-1)\Gamma(s)}\Gamma\left(\frac{1-s}{k-1}\right)\Gamma\left(\frac{ks-1}{k-1}\right)\zeta\left(\frac{ks-1}{k-1}\right)\nonumber\\
	&\quad+\frac{2^{s-1}}{\sqrt{\pi(k-1)}\Gamma^2(s)}\sum_{n=1}^{\infty} \frac{\sigma^{(k-1)}_{ks-1}(n)}{n^{\frac{ks-1}{k-1}}}\bigg\{ 
	\mathcal{G}_k\left(s,2\pi n^{\frac{1}{k-1}}\right)-\frac{\pi^{s+\frac{1}{2}}\sqrt{k-1}}{\cos\left(\frac{\pi s}{2}\right)}n^{\frac{s-1}{k-1}}\bigg\}.
\end{align}
As an application of above result, a representation for $\mathcal{L}_3(3/2)$ in terms of an infinite series of modified Bessel functions is given below.
\begin{corollary}\label{k=3 s=3/2}
Let $I_\nu(z)$ and $K_\nu(z)$ be defined in \eqref{besseli} and \eqref{kbesse} respectively. Then
	\begin{align}\label{k=3 s=3/2 eqn}
%\sum_{\substack{m, n=1}}^{\infty}\frac{1}{(m^3+mn)^\frac{3}{2}}
\sum_{n=1}^{\infty}\frac{d^{\left(\frac{1}{3}\right)}(n)}{n^{\frac{3}{2}}}&=\zeta\left(\frac{9}{2}\right)+\zeta^2\left(\frac{3}{2}\right)+\frac{1}{\sqrt{\pi}}\Gamma\left(\frac{-1}{4}\right)\Gamma\left(\frac{7}{4}\right)\zeta\left(\frac{7}{4}\right)\nonumber\\
&\quad-4\sqrt{\pi}\sum_{n=1}^{\infty}\frac{\sigma_{7/2}^{(2)}(n)}{n^{\frac{3}{2}}}\left(\pi\sqrt{n}\left(I_{\frac{1}{4}}(\pi\sqrt{n})K_{\frac{1}{4}}(\pi\sqrt{n})+3I_{-\frac{3}{4}}(\pi\sqrt{n})K_{\frac{3}{4}}(\pi\sqrt{n})\right)-2\right).
\end{align}
\end{corollary}
Also, Theorem \ref{ktz} and \eqref{lks final other} together lead to a transformation between two infinite series involving the generalized divisor functions of which one counts the number of $(k-1)$-full divisors of $n$.
\begin{corollary}\label{limiting}
Let $k>2$ be a positive integer. Let $\psi(w)=\frac{\Gamma'(w)}{\Gamma(w)}$ be the digamma function. Then
\small\begin{align}\label{general limiting case}
&\zeta^2\left(\frac{1}{k}\right)+\frac{\gamma k-\log(2\pi)+\psi\left(\frac{1}{k}\right)}{2(k-1)}+\frac{2^{1/k-1}}{\sqrt{\pi(k-1)}\Gamma^2\left(\frac{1}{k}\right)}\sum_{n=1}^{\infty}\sigma^{(k-1)}_0(n)
\bigg\{ 	\mathcal{G}_k\left(s,2\pi n^{\frac{1}{k-1}}\right)
-\frac{\pi^{\frac{1}{k}+\frac{1}{2}}\sqrt{k-1}}{n^{\frac{1}{k}}\cos\left(\frac{\pi }{2k}\right)}\bigg\}\nonumber\\
&=\frac{\zeta(2-k)}{\frac{1}{k}-1}+\frac{1}{12k}\zeta(k)-\frac{1}{2\pi^2k}\sum_{n=1}^{\infty}\frac{\sigma^{(k-1)}_{k-2}(n)}{n^2}\bigg\{\pFq{1}{2}{-\frac{1}{2}-\frac{1}{2k}}{\frac{1}{2},\frac{1}{2}-\frac{1}{2k}}{-n^2\pi^2}+(2\pi n)^{1+\frac{1}{k}}\Gamma\left(-\frac{1}{k}\right)\sin\left(\frac{\pi}{2k}\right)\bigg\}.
\end{align}
\end{corollary}
\normalsize
When we let $k=3$ in the above result, the  Meijer $G$-function appearing in the series on its resulting left-hand side can be expressed as a limiting case of a  recent generalization of the modified Bessel function studied in \cite{dkk}.  This generalization, which played an important in extending modular transformations for Eisenstein series $E_{2j}(z)$ on $\textup{SL}_2\left(\mathbb{Z}\right)$ to those for $E_a(z)$, where  $a\in\mathbb{C}$, is defined \cite[Equation (1.17)]{dkk} 
for $\nu\in\mathbb{C}\backslash\left(\mathbb{Z}\backslash\{0\}\right)$, and $z, \mu, w\in\mathbb{C}$ such that $\mu+w\neq-\frac{1}{2}, -\frac{3}{2}, -\frac{5}{2},\cdots$, by
%\backslash\{\cdots, -3, -2, -1, 1, 2, 3, \cdots\}
\begin{align}\label{def2varbessel}
	{}_{\mu}K_{\nu}(z, w)&:=\frac{\pi z^w 2^{\mu+\nu-1}}{\sin(\nu\pi)}\bigg\{\left(\frac{z}{2}\right)^{-\nu}\frac{\Gamma(\mu+w+\tfrac{1}{2})}{\Gamma(1-\nu)\Gamma(w+\tfrac{1}{2}-\nu)}\pFq12{\mu+w+\tfrac{1}{2}}{w+\tfrac{1}{2}-\nu,1-\nu}{\frac{z^2}{4}}\nonumber\\
	&\quad\quad\quad\quad\quad\quad-\left(\frac{z}{2}\right)^{\nu}\frac{\Gamma(\mu+\nu+w+\tfrac{1}{2})}{\Gamma(1+\nu)\Gamma(w+\tfrac{1}{2})}\pFq12{\mu+\nu+w+\tfrac{1}{2}}{w+\tfrac{1}{2},1+\nu}{\frac{z^2}{4}}\bigg\},
\end{align}
with 
\begin{equation*}
	{}_{\mu}K_{0}(z, w)=\lim_{\nu\to0}{}_{\mu}K_{\nu}(z, w).
\end{equation*}
Here  ${}_pF_q$ is the generalized hypergeometric function defined by
\begin{equation}\label{ghyp} {}_pF_q(a_1,\ldots,a_p;\, b_1,\ldots,b_q;\, w)
	=
	\sum_{\mu=0}^{\infty}
	\frac{\prod_{k=1}^{p} (a_k)_\mu}{\prod_{k=1}^{q} (b_k)_\mu}
	\,\frac{w^\mu}{\mu!},
\end{equation}
with $(a)_\mu=a(a+1)\cdots(a+\mu-1)$ being the shifted factorial. Thus, we have
\begin{corollary}\label{limiting k=3}
We have
\begin{align*}
&\zeta^2\left(\frac{1}{3}\right)+\frac{3\gamma-\log(2\pi)+\psi\left(\frac{1}{3}\right)}{4}+\frac{\sqrt{\pi}2^{\frac{7}{6}}}{\Gamma^2\left(\frac{1}{3}\right)}	\sum_{n=1}^{\infty}\sigma^{(2)}_0(n)\left({}_{-\frac{1}{6}}K_0(2\pi\sqrt{n},0)-\frac{2^{-\frac{5}{6}}\pi^{-\frac{1}{6}}}{\sqrt{3} n^{\frac{1}{3}}}\right)\nonumber\\
&=\frac{1}{8}+\frac{1}{36}\zeta(3)-\frac{1}{6\pi^2}\sum_{n=1}^{\infty}\frac{\sigma^{(2)}_1(n)}{n^2}\bigg\{\pFq{1}{2}{-\frac{2}{3}}{\frac{1}{2},\frac{1}{3}}{-n^2\pi^2}+\frac{1}{2}(2\pi n)^{\frac{4}{3}}\Gamma\left(-\frac{1}{3}\right)\bigg\},
\end{align*}
where  ${}_{-\frac{1}{6}}K_0(2\pi\sqrt{n},0)=\lim\limits_{s\to\frac{1}{3}}{}_{\frac{1}{2}-2s}K_{\frac{3s-1}{2}}\left(2\pi\sqrt{n},\frac{3s-1}{2}\right)$.
\end{corollary}
\begin{remark}
	While the special case $k=3$ of the Meijer $G$-function occurring in \eqref{general limiting case} can be simplified to a simpler one, that is, to  $2^{\frac{7}{6}}\pi G_{1,3}^{2, 1}\left( \begin{matrix}
		\left\{\frac{2}{3}\right\},\left\{\right\}\\
		\left\{0,0\right\}, \left\{\frac{1}{2}\right\}
	\end{matrix} \bigg | n\pi^2\right)$, it is not possible to reduce it completely in terms of well-known special functions. However, by a tedious, but straightforward limit  calculation, it can be checked that for $z>0$,
	\begin{align*}
		{}_{-\frac{1}{6}}K_0(z,0)&=\frac{2^{-\frac{7}{6}}}{\sqrt{\pi}}\Gamma\left(\frac{1}{3}\right)\bigg\{\frac{\pi\sqrt{3}}{6}+\frac{3}{2}\log(3)-2\gamma-\left(\gamma+\frac{\pi\sqrt{3}}{6}+\psi\left(\frac{1}{3}\right)+\log(\sqrt{27}z^2)\right)\pFq12{\tfrac{1}{3}}{\frac{1}{2},1}{\frac{z^2}{4}}\nonumber\\
		&\quad+\sum_{m=1}^{\infty}\frac{\left(\frac{z^2}{4}\right)^m}{(m!)^2}\frac{(1/3)_m}{(1/2)_m}\left(\psi(m+1)+2\psi(2m+1)-\psi\left(m+\frac{1}{3}\right)\right)\bigg\},
	\end{align*}
	which renders an almost closed-form evaluation of the Meijer $G$-function. 
\end{remark}
By now the reader must have noticed the obvious fact as to why we could not take $k=2$ in Theorem \ref{cs type formula}. Also, in view of \eqref{by atkinson}, it is clear that 
\begin{equation}\label{f2 atkinson}	\mathcal{F}_2(s)=\frac{1}{2}\left(\zeta^2(s)+\zeta(2s)\right).
\end{equation}	
	  Thus, having such a simple representation for $\mathcal{F}_2(s)$ in hand, a priori it seems pointless to derive an analogue of Theorem \ref{cs type formula} when $k=2$. However, once such an alternate expression is obtained, equating it with the right-hand side of \eqref{f2 atkinson} now allows us to obtain a closed-form evaluation of a series involving $\sigma_{2s-1}(n)$ and the Bessel function of the second kind, which, to the best of our knowledge, was heretofore unknown.

 We first derive an analogue of Theorem \ref{cs type formula} for $1<\textup{Re}(s)<2$. 
\begin{theorem}\label{k=2 cs type formula}
	 Let $Y_{\nu}(z)$ be defined in \eqref{ybesse}. For $1<\textup{Re}(s)<2$,
	\begin{align*}
		\mathcal{L}_2(s)&=\zeta^{2}(s)-s\zeta(s-1)\zeta(s+1)+\frac{\Gamma(1-s)\Gamma\left(2s-1\right)\zeta(2s-1)}{\Gamma(s)}\nonumber\\
		&~~-\frac{(2\pi)^s}{2\Gamma(s)\cos\left(\frac{\pi s}{2}\right)}\sum_{n=1}^{\infty}\frac{\sigma_{2s-1}(n)}{n^{s-\frac{1}{2}}}\left(\frac{1}{\sqrt{n}}+\frac{\pi}{\sqrt{2}}\frac{(-1)^n}{\sin\left(\frac{\pi s}{2}\right)}Y_{s-\frac{1}{2}}(\pi n)+\frac{s(1-s)}{2\pi n^{3/2}}\cot\left(\frac{\pi s}{2}\right)\right).
	\end{align*}
\end{theorem}
Combining \eqref{f2 atkinson} and the above theorem, we get
\begin{corollary}\label{similar to atkinson's}
	For $-1<\textup{Re}(s)<2$, 
	\begin{align}\label{similar}
		&\sum_{n=1}^{\infty}\frac{\sigma_{2s-1}(n)}{n^{s-\frac{1}{2}}}\left(\frac{1}{\sqrt{n}}+\frac{\pi}{\sqrt{2}}\frac{(-1)^n}{\sin\left(\frac{\pi s}{2}\right)}Y_{s-\frac{1}{2}}(\pi n)+\frac{s(1-s)}{2\pi n^{3/2}}\cot\left(\frac{\pi s}{2}\right)\right)\nonumber\\
		&=\frac{2\Gamma(s)\cos\left(\frac{\pi s}{2}\right)}{(2\pi)^{s}}\left\{\frac{1}{2}\zeta^{2}(s)+\frac{1}{2}\zeta(2s)-s\zeta(s-1)\zeta(s+1)+\frac{\Gamma(1-s)\Gamma\left(2s-1\right)\zeta(2s-1)}{\Gamma(s)}\right\}.
	\end{align}	
\end{corollary}
One can, of course, relax the condition $1<\textup{Re}(s)<2$  in Theorem \ref{k=2 cs type formula} by using the well-known asymptotic expansion of the Bessel function \cite[p.~202]{watson}, similar to what we did in Theorem \ref{analytically extended}, so as to meromorphically continue $\mathcal{F}_2(s)$. However, in view of \eqref{f2 atkinson}, it seems futile although one may be interested in getting identities for the series of the type in \eqref{similar} through this process. 

We record two particular cases of Corollary \ref{similar to atkinson's}.
\begin{corollary}\label{pc}
	We have
	\begin{align*}
		&\textup{(i)} 
		\sum_{n=1}^{\infty}d(n)\left((-1)^n Y_{0}(\pi n)+\frac{1}{\pi n^{\frac{1}{2}}}+\frac{1}{8\pi^2 n^{\frac{3}{2}}}\right)=\frac{1}{2\pi}\left(\gamma-\log(8\pi)+\zeta^2\left(\frac{1}{2}\right)-\zeta\left(-\frac{1}{2}\right)\zeta\left(\frac{3}{2}\right)\right),\\
		&\textup{(ii)}\sum_{n=1}^{\infty}\frac{\sigma_{2}(n)}{n}\left((-1)^nY_1(\pi n)+\frac{1}{\pi n^{\frac{1}{2}}}+\frac{3}{8\pi^2n^{\frac{3}{2}}}\right)=\frac{1}{6}-\frac{1}{8\pi^2}\left(\zeta^2\left(\frac{3}{2}\right)-\frac{3}{2}\zeta\left(\frac{1}{2}\right)\zeta\left(\frac{5}{2}\right)+\zeta(3)\right).
	\end{align*}
\end{corollary}
The first formula in the above corollary should be compared with a result of Atkinson \cite{atk}:
\begin{align*}
	\zeta^{2}\left(\frac{1}{2}\right)=\gamma-\log(8\pi)+\lim_{\delta\to 0}\left(\pi\sum_{n=1}^{\infty}d(n)(-1)^{n+1}Y_{0}(\pi n)e^{-n\delta}-\left(\frac{\pi}{8}\right)^{\frac{1}{2}}\left(\log(\delta^{-1})+\gamma-2\log(2)\right)\right).
2/1)\end{align*}
In fact, in the same paper \cite[p.~68]{atk}, Atkinson showed the ``equivalence'' of 
\begin{equation*}
	-\Gamma(1-s)2^{s-1/2}\pi^s\sum_{n=1}^{\infty}\frac{\sigma_{2s-1}(n)}{n^{s-\frac{1}{2}}}(-1)^nY_{s-\frac{1}{2}}(\pi n)
\end{equation*}
and 
\begin{equation*}
	\frac{1}{2}\left(\zeta^2(s)-\zeta(2s)-\frac{\Gamma(1-s)\Gamma(2s-1)\zeta(2s-1)}{\Gamma(s)}\right),
\end{equation*}
which should be compared with our Corollary \ref{similar to atkinson's}. However, while our result is an actual identity, Atkinson's result is in terms of what he calls ``generalized Abel summability''. See \cite{atk} for the definition of  a generalized Abel sum and other details.

\section{Preliminaries}\label{prelim}
The functional equation of the Riemann zeta function is given by \cite[p.~25]{titch}
\begin{align}\label{zetafe}
	\zeta(w)=2^{w}\pi^{w-1}\Gamma(1-w)\zeta(1-w)\sin\left(\frac{\pi w}{2}\right).	
\end{align}
The Dirichlet series of the function $\sigma^{(k)}_{z}(n)$ defined in the introduction is easily seen to be
\begin{equation}\label{gen_sigma_z}
	\sum_{n=1}^\infty \frac{\sigma^{(k)}_{s}(n)}{n^w}=\zeta(w) \zeta(k w-s) \qquad\left(\textup{Re}(w) > \max\left\{ 1, \frac{1+\textup{Re}(s)}{k}\right\}\right).
\end{equation}
%We now give the definition of the Meijer $G$-function. Suppose $m,n,p,q$ are integers such that $0\leq m \leq q$, $0\leq n \leq p$, and $a_1, \cdots, a_p$ and $b_1, \cdots, b_q$ are complex numbers such that $a_i - b_j \not\in \mathbb{N}$ for $1 \leq i \leq n$ and $1 \leq j \leq m$.   Then the Meijer $G$-function is defined by 
%\begin{align}\label{MeijerG}
%	G_{p,q}^{\,m,n} \!\left(  \,\begin{matrix} a_1,\cdots , a_p \\ b_1, \cdots , b_q \end{matrix} \; \Big| X   \right) := \frac{1}{2 \pi i} \int_L \frac{\prod_{j=1}^m \Gamma(b_j - w) \prod_{j=1}^n \Gamma(1 - a_j +w) X^w  } {\prod_{j=m+1}^q \Gamma(1 - b_j + w) \prod_{j=n+1}^p \Gamma(a_j - w)}\, dw.
%\end{align}
%Here $L$  goes from $-i \infty$ to $+i \infty$ separating the poles of $\Gamma(1-a_j+w)$  from those of $\Gamma(b_j-w)$.  The integral converges  absolutely if $p+q  < 2(m+n)$ and $|\arg(X)| < (m+n - \frac{p+q}{2}) \pi$.  In the case $p+q  = 2(m+n)$ and $\arg(X)=0$,  the integral converges absolutely if $ \left(  \Re(w) + \frac{1}{2} \right) (q-p) > \Re(\psi ) +1$,  where $\psi = \sum_{j=1}^q b_j - \sum_{j=1}^p a_j$. 

Next, we record Stirling's formula \cite[p.~224]{cop} for the Gamma function $\Gamma(z)$, $z=\sigma+it$, in a vertical strip $C\leq\sigma\leq D$ %which is instrumental in showing that the integrals along the horizontal segments of a rectangular contour for evaluating a contour integral tends to zero as the height $T$ of the contour tends to $\infty$.
\begin{equation}\label{strivert}
	|\Gamma(z)|=(2\pi)^{\tfrac{1}{2}}|t|^{\sigma-\tfrac{1}{2}}e^{-\tfrac{1}{2}\pi |t|}\left(1+O\left(\frac{1}{|t|}\right)\right)
\end{equation}
as $|t|\to\infty$. 
For $m\in\mathbb{N}, m>1$, the Gauss multiplication formula \cite[p.~52]{temme} states that
\begin{equation}\label{gmf}
	\prod_{k=1}^{m}\Gamma\left(z+\frac{k-1}{m}\right)=(2\pi)^{\frac{1}{2}(m-1)}m^{\frac{1}{2}-mz}\Gamma(mz).
\end{equation}
We also record the following well-known evaluation for $-1<\textup{Re}(s)<1$:
\begin{align}\label{sine evaluation}
	\int_{0}^{\infty}t^{s-1}\sin(t)\, dt=\Gamma(s)\sin\left(\frac{\pi s}{2}\right).
\end{align}
Let $F(z)$ and $G(z)$ be the Mellin transforms of $f(x)$ and $g(x)$ respectively. If $F(1-z)$ and $G(z)$ have a common strip of analyticity, then for any vertical line $Re(z)=c$ in the common strip, we have
\begin{align}\label{Parseval1}
	\frac{1}{2 \pi i} \int_{(c)} F(1-z)G(z) dz = \int_{0}^\infty f(t) g(t) dt,
\end{align}
under the assumption that the integral on the right-hand side exists and the conditions 
\begin{equation}\label{conditions}
	t^{c-1} g(t) \in L[0, \infty)\hspace{5mm}\text{and}\hspace{5mm}F(1-c-it) \in L(-\infty,  \infty)
\end{equation}
hold. Here, and throughout the sequel,  $\int_{(c)}$ will be used to denote the line integral $\int_{c-i\infty}^{c+i\infty}$. 

In proving Theorem \ref{integral conjecture}, which is essential to proving Theorem \ref{cs type formula}, we will be needing an extension of Parseval's formula due to Vu Kim Tuan \cite{vu} which is now discussed. This extension is useful in many instances where the usual version of Parseval's formula \cite[p.~83, Equation (3.1.11)]{kp} is inapplicable. Before we state this extension though, we need to define the concepts required  to state it, namely, a new function space and a certain class of functions. 

Denote by $\mathfrak{M}^{-1}(L)$ the space of functions $f(x)$ that are inverse Mellin transforms of functions $F(z)\in L\left(\frac{1}{2}-i\infty, \frac{1}{2}+i\infty\right)$ over the contour Re$(z)=1/2$ with norm $||f||_{\mathfrak{M}^{-1}(L)}$ equaling $\int_{0}^{\infty}\left|F\left(\frac{1}{2}+it\right)\right|\, dt$.

Let $\mathcal{K}$ be the set of functions $g(x)$ integrable on any segment $[\epsilon, E], 0<\epsilon<E<\infty$ having the property that
\begin{equation*}
	\mathfrak{M}\{g(x); z\}=G(z)=\int_{0}^{\infty}x^{z-1}g(x)\, dx,\hspace{8mm}\textup{Re}(z)=\frac{1}{2},
\end{equation*}
converges boundedly, which means that there exists a constant $C>0$ such that for almost all $\epsilon, E>0$ and $t\in\mathbb{R}$, we have $\left|\int_\epsilon^{E}x^{it-1/2}g(x)\, dx\right|<C$.

We are now ready to give the extension of Parseval's theorem \cite[Lemma 1]{vu}.
\begin{theorem}\label{vu kim tuan}
	Let $f(x)\in\mathfrak{M}^{-1}(L)$ and $g(x)\in\mathcal{K}$. Then the following convolution formula holds:
	\begin{equation*}
		\int_{0}^{\infty}f(x)g(xt)\, dx=\frac{1}{2\pi i}\int_{\left(\frac{1}{2}\right)}F(1-z)G(z)t^{-z}\, dz.
	\end{equation*}
\end{theorem}
It is not difficult to see that the above theorem implies
\begin{corollary}\label{parseval_extension_tuan}
	Let $f(x)\in\mathfrak{M}^{-1}(L)$ and $g(x)\in\mathcal{K}$. Then
	\begin{align}\label{parseval_extension_tuan_eqn}
		\int_{0}^{\infty}f\left(\frac{t}{x}\right)g(x)\, \frac{dx}{x}=\frac{1}{2\pi i}\int_{\left(\frac{1}{2}\right)}F(z)G(z)t^{-z}\, dz.
	\end{align}
\end{corollary}
\begin{remark}\label{extension_s}
	With the help of Cauchy's residue theorem, it can be seen that \eqref{parseval_extension_tuan_eqn} can be extended to any vertical strip containing the line $[1/2-i\infty, 1/2+i\infty]$ as long as it does not contain any poles of the integrand and the integrals along the horizontal segments of the rectangular contour approach zero as the height of the contour approaches infinity.
\end{remark}
\begin{remark}\label{parsevalcos}
	As mentioned in \cite[Corollary 1]{vu}, the cosine function belongs to the class $\mathcal{K}$ and hence the extension of Parseval's formula, that is, \eqref{parseval_extension_tuan_eqn} holds with $g(x)=\cos(x)$ and $f\in\mathfrak{M}^{-1}(L)$. It is this fact that will be employed in the proof of Theorem \ref{integral conjecture}.
\end{remark} 
These results are also given in \cite[p.~15-17]{yakubovich}. \\

The Bessel functions of the first and second kinds of order $\nu$ are defined by \cite[p.~40, 64]{watson}
\begin{align}
	J_{\nu}(z)&:=\sum_{m=0}^{\infty}\frac{(-1)^m(z/2)^{2m+\nu}}{m!\Gamma(m+1+\nu)} \hspace{9mm} (z,\nu\in\mathbb{C}),\nonumber\\
	Y_{\nu}(z)&:=\frac{J_{\nu}(z)\cos(\pi \nu)-J_{-\nu}(z)}{\sin{\pi \nu}}\hspace{5mm}(z\in\mathbb{C}, \nu\notin\mathbb{Z}),\label{ybesse}
\end{align}
with $Y_n(z)=\lim_{\nu\to n}Y_\nu(z)$ for $n\in\mathbb{Z}$. 
%The Bessel function of the second kind of order $\nu$ is defined for by \cite[p.~64]{watson-1944a}
The modified Bessel functions of the first and second kinds are defined by \cite[p.~77, 78]{watson}
\begin{align}
	I_{\nu}(z)&:=
	\begin{cases}
		e^{-\frac{1}{2}\pi\nu i}J_{\nu}(e^{\frac{1}{2}\pi i}z), & \text{if $-\pi<\arg(z)\leq\frac{\pi}{2}$,}\\
		e^{\frac{3}{2}\pi\nu i}J_{\nu}(e^{-\frac{3}{2}\pi i}z), & \text{if $\frac{\pi}{2}<\arg(z)\leq \pi$,}
	\end{cases}\label{besseli}\\
	K_{\nu}(z)&:=\frac{\pi}{2}\frac{I_{-\nu}(z)-I_{\nu}(z)}{\sin\nu\pi}\label{kbesse}
\end{align}
%and \cite[p.~78]{watson-1944a}
%\begin{equation}
%K_{\nu}(z):=\frac{\pi}{2}\frac{I_{-\nu}(z)-I_{\nu}(z)}{\sin\nu\pi}
%\end{equation}
respectively. When $\nu=n\in\mathbb{Z}$, $K_{n}(z)$ is interpreted as the limit $\nu\to n$ of the right-hand side of \eqref{kbesse}.

Finally, we define the Meijer $G$-function. Consider the integers $m,n,p,q$ such that $0\leq m \leq q$, $0\leq n \leq p$. Let $a_1, \cdots, a_p$ and $b_1, \cdots, b_q$ denote complex numbers such that $a_i - b_j \not\in \mathbb{N}$ for $1 \leq i \leq n$ and $1 \leq j \leq m$.   Then the Meijer $G$-function is defined by 
\begin{align}\label{MeijerG}
	G^{m, n}_{p, q}\left( \begin{matrix}
		\left\{a_1,\cdots,a_n\right\},\left\{a_{n+1}\cdots,a_p\right\}\\
		\left\{b_1,\cdots,b_m\right\}, \left\{b_{m+1},\cdots,b_q\right\}
	\end{matrix} \bigg | x\right):= \frac{1}{2 \pi i} \int_L \frac{\prod_{j=1}^m \Gamma(b_j - w) \prod_{j=1}^n \Gamma(1 - a_j +w) X^w  } {\prod_{j=m+1}^q \Gamma(1 - b_j + w) \prod_{j=n+1}^p \Gamma(a_j - w)}\, dw,
\end{align}
where the contour $L$  goes from $-i \infty$ to $+i \infty$ separating the poles of $\Gamma(1-a_j+s)$  from the poles of $\Gamma(b_j-s)$.  
The integral converges  absolutely if $p+q  < 2(m+n)$ and $|\arg(X)| < (m+n - \frac{p+q}{2}) \pi$.  In the case $p+q  = 2(m+n)$ and $\arg(X)=0$,  the integral converges absolutely if $ \left(  \re(w) + \frac{1}{2} \right) (q-p) > \re(\psi ) +1$,  where $\psi = \sum_{j=1}^q b_j - \sum_{j=1}^p a_j$.

\section{Meromorphic continuation of the Dirichlet series of Wigert's divisor function}\label{aco}
Here, we show that $\mathcal{F}_k(s)$ can be meromorphically continued to the entire $s$-complex plane.
\begin{proof}[Theorem \textup{\ref{ac}}][]

%We already know from \eqref{f2}, that for $\textup{Re}(s)>1$,
%\begin{align*}
%	\mathcal{F}_k(s)= \mathcal{L}_k(s) +\zeta(ks)
%\end{align*}
For $0<d=\textup{Re}(w)<\textup{Re}(s)$, we have  \cite[p.~196, Formula 5.36]{ober} 
\begin{align*}
	\frac{1}{(1+x)^{s}}&=	\frac{1}{2\pi i}\int_{(d)}\frac{\Gamma(s-w)\Gamma(w)}{\Gamma(s)}x^{-w}\, dw
	=\frac{1}{2\pi i}\int_{(c)}\frac{\Gamma(s+z)\Gamma(-z)}{\Gamma(s)}x^{z}\, dz,
\end{align*}
where, because of the change of variable $w=-z$, we now have $-\textup{Re}(s)<c=\textup{Re}(z)<0$. 
	
Therefore, for $-\textup{Re}(s)<c=\textup{Re}(z)<-1$, by
\eqref{l2i},
\begin{align}\label{eq1}
	\mathcal{L}_k(s)
	%&=\sum_{\substack{m=1\\ n=1}}^{\infty} \frac{1}{m^{ks} \left(1+\frac{n}{m^{k-1}}\right)^s}\nonumber \\
	&= \sum_{m,n=1}^{\infty}  \frac{1}{m^{ks}} \frac{1}{2\pi i} \int_{(c)} \frac{\Gamma(s+z) \Gamma(-z)}{\Gamma(s)} \left(\frac{n}{m^{k-1}}\right)^z dz \nonumber \\
	&=\frac{1}{2\pi i} \int_{(c)} \frac{\Gamma(s+z) \Gamma(-z)}{\Gamma(s)} \zeta(-z) \zeta(ks+(k-1)z) dz,
\end{align}
where the last step follows from interchanging the order of double sum and integration which is valid because of the exponential decay of the integrand as can be seen from Stirling's formula \eqref{strivert}. (Observe that $\textup{Re}(ks+(k-1)z)>\textup{Re}(s)>1$ as is Re$(-z)>1$.) 

To obtain the meromorphic continuation of $\mathcal{L}_k(s)$, we shift the line of integration from $\textup{Re}(z)=c$ to $\textup{Re}(z)=N+\varepsilon$ where $N$ is an odd positive integer and $\varepsilon>0$. In this process, we encounter poles of the integrand at $z=-1$ (due to $\zeta(-z)$), and at $z=0 ,1, 3, 5,\cdots, N$ (due to $\Gamma(-z)$). Let $R_a$ denote the residue of the integrand at the pole $z=a$. Then
\begin{align}\label{residue0}
R_{-1}=-\frac{\zeta(ks-(k-1))}{s-1},\hspace{4mm}R_{0}=\frac{1}{2}\zeta(ks)
\end{align}
and for $j=1, 3, \cdots, N$,
\begin{align}\label{residue1}
R_j=\frac{(-1)^{j+1} \Gamma(s+j)}{j! \Gamma(s)} \zeta(-j) \zeta(ks+(k-1)j).
\end{align}
From \eqref{eq1}, \eqref{residue0}, \eqref{residue1}, and \eqref{f2}, we observe that
\begin{align}\label{an expression}
	\mathcal{F}_k(s)&=\frac12 \zeta(ks)+\frac{\zeta(ks-(k-1))}{(s-1)}-\sum_{j=0}^{\frac{N-1}{2}} 
	\frac{  \Gamma(s+2j+1)}{(2j+1)! \Gamma(s)} \zeta(-2j-1) \zeta(ks+(k-1)(2j+1))\nonumber\\
	&\quad+\frac{1}{2\pi i} \int_{(N+\varepsilon)} \frac{\Gamma(s+z) \Gamma(-z)}{\Gamma(s)} \zeta(-z) \zeta(ks+(k-1)z) dz.
\end{align}
Since \eqref{strivert} implies that the last integral in \eqref{an expression} is holomorphic in $\textup{Re}(s)>-N-\varepsilon+1$, where $N$ is arbitrarily large, we get the meromorphic continuation of $\mathcal{F}_k(s)$ in the whole $s$-complex plane. It can be easily checked that $\mathcal{F}_k(s)$ has a double pole at $s=1$, a simple pole at $s=\frac1k$,  and simple poles at $s=\frac1k-\frac{(k-1)(2n+1)}{k}$ for $n =0, 1, 2,\cdots$.
\end{proof}

\section{An analogue of a result of Kiuchi, Tanigawa and Zhai}\label{ktz section}

%\subsection{Estimation of $\mathcal{J}(s)$}
Theorem \ref{ktz} is proved in this section.	We begin with some lemmas.
\begin{lemma}
For $Y>0, \textup{Re}(s)>0$,  and $0 < \eta=\textup{Re}(z) < 1$, define $f_s(Y)$ by
\begin{align}\label{f5}
	f_s(Y) := \frac{1}{2\pi i} \int_{(\eta)} \frac{\Gamma(s + z)}{\cos\left( \frac{\pi z}{2} \right)} Y^{-z} \, dz.
\end{align}
Then for $0< \textup{Re}(s)<1$ and $n\in\mathbb{N}$, 
\begin{align}
	f_s(2\pi n)&=\frac{2}{\pi} \Gamma(s+1)\int_{1}^{\infty} \sin(2\pi n x )x^{-s-1} dx.\label{f9}
\end{align}
%In fact \eqref{f8} can be rewritten as 
%\begin{lemma}
%	We have
%	\begin{align}\label{f88}
%		g_s(2\pi n)= s(2\pi n)^{s} \Gamma(-s) \sin \frac{\pi s}{2} +s\int_{1}^{\infty} \sin(2\pi n x )x^{-s-1} dx                       \ \mbox{for} \ 0< \Re(s)<1.
%	\end{align}
%	
\end{lemma}
%
%
%\begin{proof}
%\end{proof}
%
%
\begin{proof}
	From \cite[p.~24]{kiuchitanigawa},
	\begin{align*}
		f_s(2\pi n)&= \frac{(2\pi n)^s}{\cos\left(\frac{\pi s}{2}\right)}-\frac{2}{\pi} \Gamma(s+1)\int_{0}^{1} \sin(2\pi n x )x^{-s-1} dx.
	\end{align*}
	However, from \eqref{sine evaluation}, for $0<\textup{Re}(s)<1$,
	\begin{align*}
		\int_{0}^{\infty} \sin(2\pi n x )x^{-s-1} dx=\frac{\pi}{2\Gamma(s+1)}\frac{(2\pi n)^s}{\cos\left(\frac{\pi s}{2}\right)}.
	\end{align*}
The representation for $f_s(2\pi n)$ in \eqref{f9} now follows from the above two equations.
\end{proof}
\begin{lemma} For $\textup{Re}(s)>0$,
	\begin{align}\label{f10}
		\int_{1}^{\infty} \sin(2\pi n x )x^{-s-1} dx=\frac{1}{2\pi n}- \frac{(s+1)}{2\pi n} \int_{1}^{\infty} \cos (2\pi n x )x^{-s-2} dx.
	\end{align}
	Hence 
	\begin{align}\label{f11}
		\int_{1}^{\infty} \sin(2\pi n x )x^{-s-1} dx \ll_s \frac{1}{ n}.
	\end{align}
\end{lemma}
\begin{proof}
Use integration by parts to obtain \eqref{f10}. The bound in \eqref{f11} follows from the fact that the integral on the right-hand side of \eqref{f10} is absolutely convergent.
\end{proof}
We are now ready to prove the new representation of $\mathcal{F}_k(s)$ given in \eqref{formula3}.
\begin{proof}[Theorem \textup{\ref{ktz}}][]
	From \eqref{eq1}, for Re$(s)>1$,
\begin{align}\label{f3}
	\mathcal{L}_k(s) 
	= \frac{1}{2\pi i} \int_{(c)} 
	\frac{\Gamma(s+z)\Gamma(-z)}{\Gamma(s)} \zeta(-z) \zeta(ks + (k-1)z) \, dz,
\end{align}
where the contour of integration is the vertical line  $\textup{Re}(z) = c$ with $ -\textup{Re}(s) < c < -1$. 
%This condition ensures the convergence of the Euler product, as it implies $k\Re(s) + (k-1)c > 1$.
%The interchange of summation and integration in the last step is justified by the following estimate:
%\begin{align*}
%	\sum_{\substack{m=1 \\ n=1}}^{\infty} \frac{1}{m^{k\sigma}} 
%	\left| \int_{(c)} \frac{|\Gamma(s+z)\Gamma(-z)|}{|\Gamma(s)|} 
%	\left( \frac{n}{m^{k-1}} \right)^c dz \right|
%	\ll \zeta(-c) \zeta(k\sigma + (k-1)c) < \infty,
%\end{align*}
%where  $\sigma = \Re(s)$.
Now shift the line of integration in \eqref{f3} to the line $\textup{Re}(z) = \eta$, where $0 < \eta < 1$. In doing so, we encounter simple poles at $z = 0$ and $z = -1$ whose residues are already calculated in \eqref{residue0}. Applying the functional equation \eqref{zetafe} in the second step and invoking \eqref{gen_sigma_z} in the third, we obtain
\begin{align}\label{f4}
	\mathcal{L}_k(s) 
	&= -\frac{1}{2} \zeta(ks) 
	+ \frac{\zeta(ks - (k - 1))}{s - 1} 
	+ \frac{1}{2\pi i} \int_{(\eta)} \frac{\Gamma(s + z)\Gamma(-z)}{\Gamma(s)} \zeta(-z) \zeta(ks + (k - 1)z) \, dz \nonumber \\
	&= -\frac{1}{2} \zeta(ks) 
	+ \frac{\zeta(ks - (k - 1))}{s - 1} 
	+ \frac{1}{2} \cdot \frac{1}{2\pi i} \int_{(\eta)} \frac{\Gamma(s + z)}{\Gamma(s) \cos\left(\frac{\pi z}{2}\right)} \zeta(1 + z) \zeta(ks + (k - 1)z) (2\pi)^{-z} \, dz \nonumber \\
	&= -\frac{1}{2} \zeta(ks) 
	+ \frac{\zeta(ks - (k - 1))}{s - 1} 
	+ \frac{1}{2\Gamma(s)} \sum_{n = 1}^{\infty} \frac{\sigma_{k - 1 - ks}^{(k - 1)}(n)}{n} \cdot \frac{1}{2\pi i} \int_{(\eta)} \frac{\Gamma(s + z)}{\cos\left(\frac{\pi z}{2}\right)} (2\pi n)^{-z} \, dz.
	%&= -\frac{1}{2} \zeta(ks) 
	%	+ \frac{\zeta(ks - (k - 1))}{s - 1} 
	%	+ \frac{1}{2\Gamma(s)} \sum_{n = 1}^{\infty} \frac{\sigma_{k - 1 - ks}^{(k - 1)}(n)}{n} f_s(2\pi n),
\end{align}
 Let
\begin{align}\label{J}
	\mathcal{J}_k(s) := \frac{1}{2\Gamma(s)} \sum_{n = 1}^{\infty} \frac{\sigma_{k - 1 - ks}^{(k - 1)}(n)}{n} f_s(2\pi n),
\end{align}
where $f_s(Y)$ is defined in \eqref{f5} and $\max\{0,\frac{1-k\textup{Re}(s)}{k-1}\}<\eta<1$. This follows from the fact that $f_s(2\pi n) \ll n^{-\eta}$  and \eqref{gen_sigma_z}. Also, this bound implies that the sum in \eqref{J} converges absolutely and uniformly for $\textup{Re}(s) > 0$,
%\begin{align}
%	\sum_{n = 1}^{\infty} \frac{\sigma_{k - 1 - ks}^{(k - 1)}(n)}{n} f_s(2\pi n) 
%	&\ll \zeta(1 + \eta) \cdot \zeta(k \Re(s) + (k - 1)\eta) < \infty.
%\end{align}
 since then
\begin{align*}
	\sum_{n = 1}^{\infty} \frac{\sigma_{k - 1 - ks}^{(k - 1)}(n)}{n} f_s(2\pi n) 
	&\ll \zeta(1 + \eta) \cdot \zeta(k \textup{Re}(s) + (k - 1)\eta) \ll 1.
\end{align*}
Hence $\mathcal{J}_k(s)$ represents an analytic function of $s$ in Re$(s)>0$.

Now we would like to employ the \eqref{f9} in \eqref{J}. But the integral on the right-hand side of \eqref{f9} is valid only in $0<\textup{Re}(s)<1$. Thus,  if we initially assume  $0<\textup{Re}(s)<1$, then
\begin{align}\label{mad}
	\mathcal{J}_k(s) := \frac{s}{\pi} \sum_{n = 1}^{\infty} \frac{\sigma_{k - 1 - ks}^{(k - 1)}(n)}{n}\int_{1}^{\infty} \sin(2\pi n x )x^{-s-1} dx.
\end{align}
But \eqref{f11}, \eqref{gen_sigma_z} and the fact that $k\geq2$ implies that the series on the right-hand side of \eqref{mad} converges absolutely and uniformly for $\Re(s) > 0$, and hence represents an analytic function of $s$ in Re$(s)>0$. Since $\mathcal{J}_k(s)$ too is analytic in Re$(s)>0$ as seen before, by the principle of analytic continuation, we conclude that \eqref{mad} holds for Re$(s)>0$.

Now from \eqref{f4}, \eqref{mad} and \eqref{f2}, we arrive at 
\eqref{formula2} for Re$(s)>1$. But from the aforementioned discussion, all of the expressions involved on the right-hand side of \eqref{formula2} are analytic in Re$(s)>0$ except for simple poles at $s=1$ and $\frac{1}{k}$. However, since Theorem \ref{ac} implies that $\mathcal{L}_k(s)$ is also analytic in Re$(s)>0$ except for simple poles at $s=1$ and $\frac{1}{k}$, we obtain 
	\begin{align}\label{formula2}
	\mathcal{F}_k(s)=\frac{1}{2} \zeta(ks)+\frac{\zeta(ks-(k-1))}{(s-1)}+\frac{s}{\pi}  \sum_{n=1}^{\infty} \frac{\sigma_{k-1-ks}^{(k-1)}(n) }{n}   \int_{1}^{\infty} \sin(2\pi n x )x^{-s-1} dx.      
\end{align}
for Re$(s)>0$. Substituting \eqref{f10} in \eqref{formula2}, using \eqref{gen_sigma_z} and simplifying, we finally obtain \eqref{formula3}.
\end{proof}

\section{A Chowla-Selberg-type and Atkinson-type formula and its corollaries}\label{cs}
Two representations for $\mathcal{F}_k(s)$ were obtained in \eqref{formula3} and \eqref{an expression}.   Yet another one is derived below.
\begin{proof}[Theorem \textup{\ref{cs type formula}}][]
Consider \eqref{eq1}, where $-\textup{Re}(s)<c=\textup{Re}(z)<-1$. Now shift the line of integration from Re$(z)=c$ to Re$(z)=\lambda$, where
\begin{equation*}
	-\textup{Re}(s)-1<\lambda=\textup{Re}(z)<\frac{k}{1-k}\textup{Re}(s).
\end{equation*}
Observe that 
\begin{equation}\label{eq1.5}
\frac{k}{1-k}\textup{Re}(s)<\frac{1-k\textup{Re}(s)}{k-1}<-\textup{Re}(s).
\end{equation}
Thus, in the shifting process, we encounter a simple pole of the integrand at $z=\frac{1-ks}{k-1}$ (due to $\zeta(ks+(k-1)z$) and a simple pole at $z=-s$ (due to $\Gamma(s+z)$). Let $R_a$ denote the residue of the integrand at $z=a$. By the Cauchy residue theorem, we have
\begin{align}\label{lm residues}
	\mathcal{L}_k(s)=R_{-s}+R_{\frac{1-ks}{k-1}}+M_k(s),
\end{align}
where
\begin{align*}
	M_k(s):=	\frac{1}{2\pi i} \int_{(\lambda)} \frac{\Gamma(s+z) \Gamma(-z)}{\Gamma(s)} \zeta(-z) \zeta(ks+(k-1)z)\, dz.
\end{align*}
It is not difficult to see that
\begin{align}
	R_{-s}&=\zeta^2(s),
	\label{residues11}\\
	R_{\frac{1-ks}{k-1}}&=\frac{1}{(k-1)\Gamma(s)}\Gamma\left(\frac{1-s}{k-1}\right)\Gamma\left(\frac{ks-1}{k-1}\right)\zeta\left(\frac{ks-1}{k-1}\right).\label{residues22}
\end{align}
We now apply the functional equation \eqref{zetafe} with $w=ks+(k-1)z$ so as to get
\begin{align}\label{mks after fe}
	M_k(s)&=	2^{ks}\pi^{ks-1}\frac{1}{2\pi i} \int_{(\lambda)} \frac{\Gamma(s+z) \Gamma(-z)\Gamma(1-ks-(k-1)z)}{\Gamma(s)}\sin\left(\frac{\pi}{2}(ks+(k-1)z)\right)\nonumber\\
	&\qquad\qquad\qquad\qquad\times\zeta(-z)\zeta(1-ks-(k-1)z)(2\pi)^{(k-1)z} \, dz.
\end{align}
Invoking \eqref{gen_sigma_z}, we see that
\begin{align}\label{gen_sigma_z our case}
	\zeta(-z)\zeta(1-ks-(k-1)z)=\sum_{n=1}^{\infty} \frac{\sigma^{(k-1)}_{ks-1}(n)}{n^{-z}},
\end{align}
where Re$(z)<\min{\{-1,\frac{k}{1-k}\textup{Re}(s)\}}=\frac{k}{1-k}\textup{Re}(s)$, because of \eqref{eq1.5}. Therefore, substituting \eqref{gen_sigma_z our case} in \eqref{mks after fe} and interchanging the order of summation and integration, we see that
\begin{align}\label{mks after fe1}
	M_k(s)
	&=\frac{(2\pi)^{ks}}{\pi\Gamma(s)}\sum_{n=1}^{\infty} \sigma^{(k-1)}_{ks-1}(n)I_{n,k},
\end{align}
where
\begin{equation}\label{ink}
I_{n,k}:=\frac{1}{2\pi i}\int_{(\lambda)} F(z)G(z)\left((2\pi)^{(k-1)}n\right)^z \, dz,	
\end{equation}
and
\begin{align*}
	F(z)&=\frac{\Gamma(s+z) \Gamma(-z)}{\Gamma(s)},\\
	G(z)&=\Gamma(1-ks-(k-1)z)\sin\left(\frac{\pi}{2}(ks+(k-1)z)\right).
\end{align*}
One may now be tempted to use Parseval's formula \eqref{Parseval1} to evaluate \eqref{ink}, however, it is inapplicable since the first condition in \eqref{conditions} does not hold. One could have resorted to Theorem \ref{vu kim tuan}, however, as can be checked, its requirement that Re$(z)=1/2$ should lie in the intersection of the vertical strips of validity of the Mellin transforms $F(z)$ and $G(z)$ is not satisfied. This hurdle is circumvented as follows.

We shift the line of integration to the right from Re$(z)=\lambda$ to $Re(z)=c', c'>0$,  by constructing a rectangular contour $[\l-iT,c'-iT, c'+iT,\l+iT]$, then let $T\to\infty$, use \eqref{strivert} to show that the integrals along the horizontal segments tends to zero in the limit, then let $c'\to\infty$, show that the integral over the shifted line of integration also tends to zero in this limit, and finally consider the contribution of the poles of the integrand at $z=j, j\geq0$ (due to $\Gamma(-z)$), at $z=\frac{2m+1-ks}{k-1}, m\geq0$ (due to $\Gamma(1-ks-(k-1)z)$), and at $z=-s$ (due to $\Gamma(s+z)$) with the residues 
\begin{align*}
R_j&=-\frac{(-(2\pi)^{(k-1)}n)^j}{j!\Gamma(s)}\Gamma(j+s)\Gamma(1+j-k(s+j))\sin\left(\frac{\pi}{2}(ks+(k-1)j)\right),\nonumber\\
R_{\frac{2m+1-ks}{k-1}}&=\frac{(-1)^{m+1}}{(k-1)(2m)!\Gamma(s)}((2\pi n^{\frac{1}{k-1}})^{1+2m-ks}\Gamma\left(\frac{1+2m-s}{k-1}\right)\Gamma\left(\frac{ks-2m-1}{k-1}\right),\nonumber\\
R_{-s}&=(2\pi)^{-(k-1)s}n^{-s}\Gamma(1-s)\sin\left(\frac{\pi s}{2}\right).
\end{align*}
so that by Cauchy's residue theorem, we have
\begin{align*}
I_{n,k}&=\sum_{j=0}^{\infty}\frac{(-(2\pi)^{(k-1)}n)^j}{j!\Gamma(s)}\Gamma(j+s)\Gamma(1+j-k(s+j))\sin\left(\frac{\pi}{2}(ks+(k-1)j)\right)\nonumber\\
&\quad+\sum_{m=0}^{\infty}\frac{(-1)^{m}}{(k-1)(2m)!\Gamma(s)}((2\pi n^{\frac{1}{k-1}})^{1+2m-ks}\Gamma\left(\frac{1+2m-s}{k-1}\right)\Gamma\left(\frac{ks-2m-1}{k-1}\right)\nonumber\\
&\quad-(2\pi)^{-(k-1)s}n^{-s}\Gamma(1-s)\sin\left(\frac{\pi s}{2}\right)\nonumber\\
&=
\frac{(2\pi n^{\frac{1}{k-1}})^{1-ks}2^{s-2}}{\sqrt{\pi(k-1)}\Gamma(s)}\mathcal{G}_k(s,y)-(2\pi)^{-(k-1)s}n^{-s}\Gamma(1-s)\sin\left(\frac{\pi s}{2}\right),
\end{align*}
where we employed Theorem \ref{meijerg specialization}, proved in the Appendix, with $y=2\pi n^{\frac{1}{k-1}}$. Now with the same expression for $y$, invoking Theorem \ref{integral conjecture}, we arrive
\begin{equation}\label{ink expression}
I_{n,k}=\frac{(2\pi n^{\frac{1}{k-1}})^{1-ks}}{k-1}\int_{0}^{\infty}\left((1+x)^{-s}-1\right)x^{\frac{s+k-2}{1-k}}\cos\left(2\pi (nx)^{\frac{1}{k-1}}\right)\, dx.
\end{equation}

Theorem \ref{cs type formula} now follows from \eqref{f2}, \eqref{lm residues}, \eqref{residues11}, \eqref{residues22}, \eqref{mks after fe1} and \eqref{ink expression}. 

\end{proof}

\subsection{An integral evaluation}\label{int_eval}
 Theorem \ref{integral conjecture} is proved here. Part of the proof of this evaluation is broken down into two lemmas. We begin with the first.

\begin{lemma}\label{lemma:7.1}
Let $k\in\mathbb{N}, k\geq2,$ and $\textup{Re}(u)>0$. Let $0<\textup{Re}(s) <1 + (k-1)\textup{Re}(\nu)$. Then
\begin{align}\label{lemma:7.1 eqn}
	&\int_{0}^{\infty} \frac{t^{s-1}\cos(yt)}{(u^{k-1}+t^{k-1})^{\nu}}\,dt\nonumber\\
	&=
	\frac{2^{\nu-2} u^{\,s-(k-1)\nu}}{\sqrt{\pi(k-1)}\,\Gamma(\nu)}G^{k+1,2}_{2,2k}\left( \begin{matrix}
		\left\{\frac{2k-2-s}{2k-2}, \frac{k-1-s}{2k-2}\right\},\left\{\right\}\\
		\left\{0,\frac{1}{k-1},\cdots,\frac{k-2}{k-1}, \frac{\nu}{2}-\frac{s}{2k-2}, \frac{1+\nu}{2}-\frac{s}{2k-2}\right\}, \left\{\frac{1}{2k-2},\frac{3}{2k-2},\cdots,\frac{2k-3}{2k-2}\right\}
	\end{matrix} \bigg | \left(\tfrac{uy}{2k-2}\right)^{2k-2} \right).
\end{align}
\end{lemma}
\begin{proof}
Let
\begin{align}\label{IBP1}
	I := \int_{0}^{\infty} \frac{t^{s-1}\cos(yt)}{(u^{k-1}+t^{k-1})^{\nu}}\, dt.
\end{align}
We first show that $I$ converges for $0<\textup{Re}(s) < 1 + (k-1)\textup{Re}(\nu)$.
Consider
\begin{align*}
	I_{\varepsilon,M} := \int_{\varepsilon}^{M} \frac{t^{s-1}\cos(yt)}{(u^{k-1}+t^{k-1})^{\nu}}\, dt.
\end{align*}
Applying integration by parts, we have
\begin{align}\label{IBPOO}
	I_{\varepsilon,M} &= 
	\left[ \frac{t^{s-1}}{(u^{k-1}+t^{k-1})^{\nu}} \frac{\sin(yt)}{y} \right]_{\varepsilon}^{M} 
	- \frac{1}{y} \int_{\varepsilon}^{M} 
	\Biggl(
	\frac{(s-1)t^{s-2}}{(u^{k-1}+t^{k-1})^{\nu}}
	- \frac{\nu(k-1) t^{s+k-3}}{(u^{k-1}+t^{k-1})^{\nu+1}}
	\Biggr) \sin(yt) \, dt.
\end{align}

Employing the fact $\lim_{u\rightarrow 0^+} \frac{\sin u}{u}=1$ and $\textup{Re}(s)<1+(k-1)\textup{Re}(\nu)$, we see
\begin{align*}
	\lim_{\varepsilon\rightarrow 0^+} \frac{\varepsilon^{\textup{Re}(s)-1}}{(u^{k-1}+\varepsilon^{k-1})^{\textup{Re}(\nu)}} \frac{\sin(y\varepsilon)}{y} &=0, \quad \text{if } \textup{Re}(s)>0,
	\\
	\lim_{M\rightarrow \infty} \frac{M^{\textup{Re}(s)-1}}{(u^{k-1}+M^{k-1})^{\textup{Re}(\nu)}} \frac{|\sin(yM)|}{y} &\leq \frac{1}{|y|} \lim_{M\rightarrow \infty} M^{\textup{Re}(s)-1-(k-1)\textup{Re}(\nu)} = 0,
\end{align*}
provided $\textup{Re}(s) < 1 + (k-1)\textup{Re}(\nu)$.
Similarly, one can show the integrals in \eqref{IBPOO} converges as $\varepsilon \to 0^+$ and $M \to \infty$.  Therefore the integral in \eqref{IBP1} converges whenever $0<\textup{Re}(s) < 1 + (k-1)\textup{Re}(\nu)$.

Now let 
\begin{equation*}
	f(t):=\frac{t^{s-1}}{(u^{k-1}+t^{k-1})^{\nu}}\hspace{5mm}\text{and}\hspace{5mm} g(t):=\cos(yt).
\end{equation*}
Let $F(w)$ and $G(w)$ denote the Mellin transforms of $f$ and $g$ respectively. Using the Euler's beta integral formula and the well-known Mellin transform of the cosine function, we see that
	\begin{align*}
	F(w) &= \int_0^\infty t^{w-1} f(t) \, dt
	= \frac{u^{w+s-1-(k-1)\nu}}{k-1} \frac{\Gamma\!\left(\frac{w+s-1}{k-1}\right) \Gamma\!\left(\nu - \frac{w+s-1}{k-1}\right)}{\Gamma(\nu)},
\end{align*}
provided $\ 1-\textup{Re}(s)<\textup{Re}(w)<1-\textup{Re}(s)+(k-1)\textup{Re}(\nu)$,
and
\begin{align}\label{MT2}
	G(w) &= \int_0^\infty t^{w-1} g(t) \, dt
	= y^{-w} \Gamma(w) \cos\!\left(\frac{\pi w}{2}\right),
\end{align}
provided $0 < \textup{Re}(w) < 1$.

We would now like to apply the extended Parseval's formula in Theorem \ref{vu kim tuan} keeping Remark \ref{parsevalcos} in mind. However, this requires that  
$\textup{Re}(w)=\frac12$ lie in the region
\begin{align*}
	\max\{0, \textup{Re}(s)-(k-1) \textup{Re}(\nu)\}< \textup{Re}(w)< \min\{\textup{Re}(s), 1\}.
\end{align*}
which in turn implies
		$\frac12< \textup{Re}(s)<\frac12+(k-1)\textup{Re}(\nu).$ (Note that this also implies Re$(\nu)>0$.)
		
Thus, invoking Theorem \ref{vu kim tuan} and employing the change of variable $w=2(k-1)\xi$ in the second step, we see that for $\frac12< \textup{Re}(s)<\frac12+(k-1)\textup{Re}(\nu)$,
	\begin{align*}
	I &= \frac{u^{s-(k-1)\nu}}{2\pi i (k-1) \Gamma(\nu)}
	\int_{(\frac12)} (uy)^{-w} \Gamma(w) \cos\!\left(\frac{\pi w}{2}\right) 
	\Gamma\!\left(\frac{s-w}{k-1}\right) \Gamma\!\left(\nu - \frac{s-w}{k-1}\right) \, dw.\\
	&=\frac{2 \pi u^{s-(k-1)\nu}}{2\pi i \Gamma(\nu)}\int_{\left(\frac{1}{4k-4}\right)} \frac{\Gamma((2k-2)\xi) \Gamma\left(\frac{s}{k-1}-2\xi\right) \Gamma\left(\nu-\frac{s}{k-1}+2\xi\right)}{\Gamma \left(\frac12-(k-1)\xi\right) \Gamma\left(\frac12+(k-1)\xi\right)} ((uy)^{2k-2})^{-\xi} d\xi\\
	&=\frac{\sqrt{ \pi} u^{s-(k-1)\nu}}{2\pi i \Gamma(\nu)} \int_{\left(\frac{1}{4k-4}\right)} \frac{\Gamma((k-1)\xi) \Gamma\left(\frac{s}{k-1}-2\xi\right) \Gamma\left(\nu-\frac{s}{k-1}+2\xi\right)}{\Gamma \left(\frac12-(k-1)\xi\right) } \left((\tfrac{uy}{2})^{2k-2}\right)^{-\xi} d\xi,
\end{align*}
where in the last step we used the duplication formula for Gamma function. Now using the Gauss multiplication formula \eqref{gmf} multiple times to simplify the integrand, we finally arrive at 
	\begin{align}\label{intef}
	I&=\frac{2^{\nu-2}u^{s-(k-1)\nu}}{2\pi i\sqrt{\pi(k-1)}\Gamma(\nu)}
	\int_{(\frac{1}{4k-4})}
	\frac{
		\Gamma\!\left(\frac{s}{2k-2}-\xi\right)
		\Gamma\!\left(\frac12+\frac{s}{2k-2}-\xi\right)
		\prod_{j=0}^{k-2}\Gamma\!\left(\xi+\frac{j}{k-1}\right)
			}{\prod_{j=1}^{k-1}\Gamma\!\left(\frac{2j-1}{2k-2}-\xi\right)}
	\nonumber \\[0.5em]
	&\quad\times \Gamma\!\left(\xi-\tfrac{s}{2k-2}+\tfrac{\nu}{2}\right)
	\Gamma\!\left(\xi-\tfrac{s}{2k-2}+\tfrac{1+\nu}{2}\right)
	\left(\left(\tfrac{uy}{2k-2}\right)^{2k-2}\right)^{-\xi} d\xi
\end{align}
%	\begin{align}\label{intef}
%	I&=\frac{2^{\nu-2}u^{s-(k-1)\nu}}{2\pi i\sqrt{\pi}\Gamma(\nu)\sqrt{k-1}}
%	\int_{(\frac{1}{4k-4})}
%	\frac{
%		\Gamma\!\left(\frac{s}{2k-2}-\xi\right)
%		\Gamma\!\left(\frac12+\frac{s}{2k-2}-\xi\right)
%		\prod_{j=0}^{k-2}\Gamma\!\left(\xi+\frac{j}{k-1}\right)
%	}{}
%	\nonumber \\[0.5em]
%	&\quad\times
%	\frac{
%		\Gamma\!\left(\xi-\frac{s}{2k-2}+\frac{\nu}{2}\right)
%		\Gamma\!\left(\xi-\frac{s}{2k-2}+\frac{1+\nu}{2}\right)
%	}{
%		\prod_{j=1}^{k-1}\Gamma\!\left(\frac{2j-1}{2k-2}-\xi\right)
%	}
%	\left(\left(\frac{uy}{2k-2}\right)^{2k-2}\right)^{-\xi} d\xi
%\end{align}
for $\frac12< \textup{Re}(s)<\frac12+(k-1)\textup{Re}(\nu).$  

Now the integral $I$ in \eqref{IBP1} converges uniformly and is analytic in the region $0< \textup{Re}(s)<1+(k-1)\textup{Re}(\nu)$. Also, the right-hand side of  \eqref{intef} converges uniformly for $\textup{Re}(s)>0$ and is analytic there. Therefore, by the principle of analytic continuation we see that \eqref{intef} holds for $0< \textup{Re}(s)<1+(k-1)\textup{Re}(\nu)$. Finally the proof of Lemma \ref{lemma:7.1} is complete upon noticing that the right-hand side of \eqref{intef} is nothing but the right-hand side of \eqref{lemma:7.1 eqn} using the definition of Meijer $G$-function given in \eqref{MeijerG}.
\end{proof}
\begin{remark}
The two special cases of the above lemma for $k=2$ and $k=3$ are contained in the literature, namely, in \cite[p.~392, Formula \textbf{2.5.7.8}]{Prudnikov} and \cite[p.~43, Formula (5.10)]{ober} respectively. For example, when $k=3$, using Slater's theorem \cite[p.~145]{Luke}, we obtain \cite[p.~43, Formula (5.10)]{ober}, namely, for $0<\textup{Re}(s)<1+2\textup{Re}(\nu),$
\begin{align*}
	\int_{0}^{\infty} \frac{t^{s-1} \cos (yt) \ dt}{(u^{2}+t^{2})^{\nu}}&=\frac{u^{s-2\nu}}{2}\frac{\Gamma\left(\frac{s}{2}\right)\Gamma\left(\nu-\frac{s}{2}\right)}{\Gamma(\nu)}{}_1F_{2}\left( \begin{matrix}
		\frac{s}{2} \\
		1-\nu+\frac{s}{2},\frac{1}{2}
	\end{matrix}~;\frac{u^2y^2}{4}\right)\nonumber\\
	&\quad+\frac{\sqrt{\pi}}{2}\left(\frac{y}{2}\right)^{2\nu-s}\frac{\Gamma\left(\frac{s}{2}-\nu \right)}{\Gamma\left(\frac{1}{2}+\nu-\frac{s}{2}\right)}{}_1F_{2}\left( \begin{matrix}
		\nu \\
		\frac{1}{2}+\nu-\frac{s}{2},1+\nu-\frac{s}{2}
	\end{matrix}~;\frac{u^2y^2}{4}\right).
\end{align*}
However, in general, for any natural number $k>1$, the evaluation of the integral in the above lemma was missing in the literature.
\end{remark}

\begin{lemma}\label{MImp thm}
	Let $\textup{Re}(\nu)>0$. 
	If $k=2$ and $-1<\textup{Re}(s)<0$, or $k>2$ and $-2<\textup{Re}(s)<0$, we have
	\begin{align}\label{MImp}
		&\int_{0}^{\infty}\left(\frac{\cos (yt)}{(u^{k-1}+t^{k-1})^{\nu}}-\frac{1}{u^{(k-1)\nu}}\right)t^{s-1}\, dt\nonumber\\
		&=\frac{2^{\nu-2} u^{\,s-(k-1)\nu}}{\sqrt{\pi(k-1)}\,\Gamma(\nu)}G^{k+1,2}_{2,2k}\left( \begin{matrix}
			\left\{\frac{2k-2-s}{2k-2}, \frac{k-1-s}{2k-2}\right\},\left\{\right\}\\
			\left\{0,\frac{1}{k-1},\cdots,\frac{k-2}{k-1}, \frac{\nu}{2}-\frac{s}{2k-2}, \frac{1+\nu}{2}-\frac{s}{2k-2}\right\}, \left\{\frac{1}{2k-2},\frac{3}{2k-2},\cdots,\frac{2k-3}{2k-2}\right\}
		\end{matrix} \bigg | \left(\tfrac{uy}{2k-2}\right)^{2k-2} \right).
	\end{align}
\end{lemma}
\begin{proof}
	For brevity, let $G(u,y,s,k, \nu)$ denote the Meijer $G$-function in \eqref{lemma:7.1 eqn}. 
	Applying the Mellin inversion theorem to the integral evaluation in Lemma \ref{lemma:7.1}, we find that 
\begin{align*}
	A_{y}(u,t):=\frac{\cos (yt)}{(u^{k-1}+t^{k-1})^{\nu}}=\frac{2^{\nu-2}u^{-(k-1)\nu}}{\sqrt{\pi(k-1)}\Gamma(\nu)}\frac{1}{2\pi i}\int_{(c)}  G(u,y,s,k, \nu)\left( \frac{t}{u}\right)^{-s}\, ds,
\end{align*}
where $0<c=\textup{Re}(s)<1+(k-1)\textup{Re}(\nu).$
Further invoking Slater's formula \cite[p.~145, Equation (7)]{Luke} for the Meijer $G$-function, we get 
\begin{align*}
	A_{y}(u,t)=\frac{2^{\nu-2}u^{-(k-1)\nu}}{\sqrt{\pi(k-1)}\Gamma(\nu)}\frac{1}{2\pi i}&\int_{(c)}\sum_{h=1}^{k+1}\frac{\prod_{j=1}^{k=1}\Gamma(b_j-b_h)^{*}\Gamma(b_h+\frac{s}{2k-2})\Gamma(\frac{1}{2}+b_h+\frac{s}{2k-2})}{\prod_{j=k+2}^{2k}\Gamma(1+b_h-b_j)}\left( \frac{uy}{2k-2}\right)^{(2k-2)b_{h}}\nonumber\\
	&\times{}_2F_{2k-1}\left(\begin{matrix}
		b_h+\frac{s}{2k-2},\frac{1}{2}+b_h+\frac{s}{2k-2}  \\
		1+b_h-b^{*}_{2k}
	\end{matrix}\, \Bigg| \left(\frac{-uy}{2k-2}\right)^{(2k-2)} \right)\left( \frac{t}{u}\right)^{-s}\, ds.
\end{align*}	
We now shift the line of integration from $\textup{Re}(s)=c$ to $\textup{Re}(s)=\l$, where $-1<\l<0$ if $k=2$ and $-2<\l<0$ if $k>2$. In this process, we encounter a simple pole of the integrand at $s=0$ (due to the term corresponding to $h=1$), the residue at which can be easily calculated to be $u^{-(k-1)\nu}$. Then, by Cauchy's residue theorem,
	\begin{align*}
	A_{y}(u,t)-\frac{1}{u^{(k-1)\nu}}=\frac{2^{\nu-2}u^{-(k-1)\nu}}{\sqrt{\pi(k-1)\Gamma(\nu)}}\frac{1}{2\pi i}\int_{(\l)}G(u,y,s, k, \nu)\left( \frac{t}{u}\right)^{-s}\, ds.
\end{align*}
Finally, applying the Mellin inversion theorem again, we arrive at \eqref{MImp}.
\end{proof}
Armed with Lemmas \ref{lemma:7.1} and \ref{MImp thm}, we are now ready to prove Theorem \ref{integral conjecture}.
\begin{proof}[Theorem \textup{\ref{integral conjecture}}][]
Replace $s$ by $1-s$ and then let $\nu=s$ in Lemma \ref{MImp thm} so that for $1<\textup{Re}(s)<3$,
	\begin{align}\label{Impeq1}
	\int_{0}^{\infty} \left(\cos (yt)(1+t^{k-1})^{-s}-1\right)t^{-s}\, dt
	=\frac{2^{s-2}}{\sqrt{\pi(k-1)}\Gamma(s)}\mathcal{G}_k\left(s,y\right),
\end{align}
where $\mathcal{G}_k(s,y)$ is defined in \eqref{short notation}. Using \eqref{MT2}, it is easy to see that for $1<\textup{Re}(s)<3$,
\begin{align}\label{Impeq2}
	\int_{0}^{\infty} t^{-s} (\cos(yt)-1) dt = y^{s-1} \Gamma(1-s) \sin\left(\frac{\pi s}{2}\right).   
\end{align}
Equation \eqref{integral conjecture eqn} now follows  for $1<\textup{Re}(s)<3$ from \eqref{Impeq1} and \eqref{Impeq2} since
\begin{align*}
	\int_{0}^{\infty} t^{-s}  ( (1+t^{k-1})^{-s} -1) \cos(yt)dt
	=\int_{0}^{\infty} \left(\cos (yt)(1+t^{k-1})^{-s}-1\right)t^{-s}\, dt- \int_{0}^{\infty} t^{-s} (\cos(yt)-1) dt,
\end{align*}
and because
\begin{align*}
\int_{0}^{\infty}\left((1+x)^{-s}-1\right)x^{\frac{s+k-2}{1-k}}\cos\left(yx^{\frac{1}{k-1}}\right)\, dx=(k-1)	\int_{0}^{\infty} t^{-s}  ( (1+t^{k-1})^{-s} -1) \cos(yt)dt.
\end{align*}
\end{proof}

\subsection{Asymptotic expansion of a Meijer $G$-function}\label{asymp-meijer}

In order to prove Theorem \ref{analytically extended}, we first need to derive the asymptotic expansion of the Meijer $G$-function occurring in Theorem \ref{integral conjecture}. This is done in this subsection.

Let $m, n, p$ and $q$ be integers satisfying $1 \le n \le p < q$ and $1 \le m \le q$. Let 
\begin{align*}
	a_j - b_h &\neq 1,2,3,\ldots,
	\quad j=1,2,\ldots,n,\quad h=1,2,\ldots,m,\nonumber\\
	a_j - a_t &\neq 0,\pm1,\pm2,\ldots,
	\quad j,t=1,2,\ldots,n,\quad j \neq t,
\end{align*}
Then, as $|z| \to \infty$, 	where $|\arg(z)| \le \rho\pi - \delta$ with
$\rho = q-p > 0$ and $\delta \ge 0$, we have \cite[p.~179, Theorem~2]{Luke}
\begin{align*}
G^{m, n}_{p, q}\left( \begin{matrix}
	\left\{a_1,\cdots,a_n\right\},\left\{a_{n+1}\cdots,a_p\right\}\\
	\left\{b_1,\cdots,b_m\right\}, \left\{b_{m+1},\cdots,b_q\right\}
\end{matrix} \bigg | z\right)
	\sim
	\sum_{j=1}^{n}
	\exp\!\bigl(-i\pi(\nu+1)a_j\bigr)\,
	\Delta^{m,n}_{q}(j)\,
	E_{p,q}\!\left(
	z\, e^{i\pi(\nu+1)} || a_j
	\right),
\end{align*}
where $\nu := q - m - n$, and
\begin{align}
	\Delta^{m,n}_{q}(j)
	&:=
	(-1)^{\nu+1}
	\frac{
		\prod_{\substack{\ell=1 \\ \ell \neq j}}^{n}
		\Gamma(a_j-a_\ell)\,\Gamma(1+a_\ell-a_j)
	}{
		\prod_{\ell=m+1}^{q}
		\Gamma(a_j-b_\ell)\,\Gamma(1+b_\ell-a_j)
	}, \nonumber\\
	E_{p,q}(w \mid a_j)
	&:=
	w^{a_j-1}
	\frac{\prod_{\ell=1}^{q}\Gamma(1+b_\ell-a_j)}
	{\prod_{\ell=1}^{p}\Gamma(1+a_\ell-a_j)}
	\nonumber \\
	&\quad \times
	{}_{q}F_{p-1}\!\left(
	\begin{matrix}
		1 + b_1 - a_j,\ \ldots,\ 1 + b_q - a_j \\[4pt]
		1 + a_1 - a_j,\ \ldots,\ 1 + a_{j-1} - a_j,\ 1 + a_{j+1} - a_j,\ \ldots,\ 1 + a_p - a_j
	\end{matrix}\;\middle|\; -\dfrac{1}{w} \right), \label{epq}
\end{align}
where  ${}_pF_q$ is defined in \eqref{ghyp}.

Recall that $a_1 = \frac{2k-3+s}{2k-2}, 
	a_2 = \frac{k-2+s}{2k-2}$,
and that
$	b_i = \frac{i-1}{k-1}$  for $1 \le i \le k-1, $
	$b_k = \frac{sk-1}{2k-2},	b_{k+1} = \frac{sk-1}{2k-2} + \frac{1}{2}$,
and 
	$b_{k+1+t} = \frac{2k-(2t+1)}{2k-2}$ for $1 \le t \le k-1$. We first simplify $\Delta^{k+1,2}_{2k}(1)$ and $\Delta^{k+1,2}_{2k}(2)$. By the definition,
 \begin{align*}
 	\Delta^{k+1,2}_{2k}(1) & =(-1)^{k-2}  \frac{
 		\Gamma(a_1-a_2)\,\Gamma(1+a_2-a_1)
 	}{
 		\prod_{\ell=k+2}^{2k}
 		\Gamma(a_1-b_\ell)\,\Gamma(1+b_\ell-a_1)} \nonumber \\
 	%&=\frac{ (-1)^{k} \pi}{ \prod_{\ell=k+2}^{2k}
 	%	\Gamma(a_1-b_\ell)\,\Gamma(1+b_\ell-a_1)}\nonumber\\
 	&=\frac{ (-1)^{k} \pi}{ \prod_{t=1}^{k-1}
 		\Gamma\left(\frac{s-2+2t}{2k-2} \right)\,\prod_{t=1}^{k-1}
 		\Gamma\left(\frac{2t-s}{2k-2} \right)} \nonumber\\
 	&=\frac{ (-1)^{k} \pi}{ \prod_{t=1}^{k-1}
 		\Gamma\left(\frac{s}{2k-2} +\frac{t-1}{k-1}\right)\,\prod_{t=1}^{k-1}
 		\Gamma\left(\frac{2-s}{2k-2}+\frac{t-1}{k-1} \right)},
 \end{align*}
 where, in the last step, we replaced $t$ by $k-t$ in the second product in the denominator.
Applying \eqref{gmf} 
and simplifying, we arrive at
  	\begin{align*}
  	\Delta^{k+1,2}_{2k}(1) 
  	&= \frac{ (-1)^{k} \pi}{ (2\pi)^{k-2}   \Gamma\left(\frac{s}{2}\right) \Gamma \left(1-\frac{s}{2}\right)}= (-1)^{k} (2\pi)^{2-k}  \sin \left(\frac{\pi s}{2}\right).
  \end{align*}
Similarly,
	\begin{align*}
	\Delta^{k+1,2}_{2k}(2)  
	= \frac{ (-1)^{k+1} \pi}{ (2\pi)^{k-2}   \Gamma\left(\frac12+\frac{s-k}{2}\right) \Gamma \left(\frac12-\frac{s-k}{2}\right)}
	=  (-1)^{k+1} (2\pi)^{2-k} \cos \left(\frac{\pi (s-k)}{2}\right).
\end{align*}
We proceed to obtain simplified expressions for $E_{2,2k}\!\big(z e^{i\pi(k-2)} \mid a_1\big)$ and $E_{2,2k}\!\big(z e^{i\pi(k-2)} \mid a_2\big)$. To that end, note that using \eqref{gmf} repeatedly, for any $\mu\in\mathbb{N}\cup\{0\}$, we have 
\begin{align}
	\prod_{\ell=1}^{2k}\Gamma(1+b_\ell-a_1+\mu)&=2^{\,k-1-2\mu k}\,\pi^{k-1}(k-1)^{s -\frac{1}{2} -2\mu(k-1)}\,
	\Gamma(s+2\mu)\,
	\Gamma\!\big(1 - s + 2\mu(k-1)\big)\,
	,\label{a1mu}\\
	\prod_{\ell=1}^{2k}\Gamma(1+b_\ell-a_2+\mu)&=2^{\,-1-2\mu k}\,\pi^{k-1}\,
	(k-1)^{s-k +\frac{1}{2} -2\mu(k-1)}\Gamma(s+2\mu+1)\,
	\Gamma\!\big(k - s + 2\mu(k-1)\big)\,
	. \label{a2mu}
	\end{align}
From \eqref{epq}, \eqref{ghyp}, \eqref{a1mu} and \eqref{a2mu}, we find, after a lot of simplification, that
\small\begin{align*}
	E_{2,2k}\!\big(z e^{i\pi(k-2)} \mid a_1\big)
	& =\Big(z e^{i\pi(k-2)}\Big)^{\frac{s-1}{2k-2}}   \pi^{k-\frac32} 2^{k-1} (k-1)^{s-\frac12} \nonumber\\
	& \quad\times\left(
	\frac{\pi}{\sin \pi s}+\sum_{\mu=1}^{M}
	\,
	\frac{\Gamma(s+2\mu)\,
		\Gamma\!\big(1 - s + 2\mu(k-1)\big)}{2^{2\mu k}\,
		(k-1)^{2\mu(k-1)} (\frac12)_{\mu}\mu!(-z e^{i\pi(k-2)})^{\mu}}
	+O_k\left( |z|^{-(M+1)} \right)
	\right),\\
	E_{2,2k}\!\big(z e^{i\pi(k-2)} \mid a_2\big)
	& =\Big(z e^{i\pi(k-2)}\Big)^{\frac{s-k}{2k-2}}   \pi^{k-\frac32}  (k-1)^{s-k+\frac12}  \nonumber\\
	& \quad\times\left(
	\Gamma(s+1) \Gamma(k-s)+\sum_{\mu=1}^{M}
	\,
	\frac{\Gamma(s+2\mu+1)\,
		\Gamma\!\big(k - s + 2\mu(k-1)\big)}{2^{2\mu k}\,(k-1)^{2\mu(k-1)} (\frac32)_{\mu}\mu!(-z e^{i\pi(k-2)})^{\mu}}
	+O_k\left( |z|^{-(M+1)} \right)
	\right).
\end{align*}
\normalsize
%It is not difficult to see that with $z=\left(y/(2k-2)\right)^{2k-2}$, the first term   with $y=2\pi n^{1/(k-1)}$, we have
Therefore, 
\begin{align*}
&G_{2, \, \,\,\,\,\,2k}^{k+1, \, \, 2}\left( \begin{matrix}
		\left\{\frac{2k-3+s}{2k-2}, \frac{k-2+s}{2k-2}\right\},\left\{\right\}\\
		\left\{0,\frac{1}{k-1},\cdots,\frac{k-2}{k-1}, \frac{sk-1}{2k-2}, \frac{sk-1}{2k-2}+\frac{1}{2}\right\}, \left\{\frac{1}{2k-2},\frac{3}{2k-2},\cdots,\frac{2k-3}{2k-2}\right\}
	\end{matrix} \bigg | z\right)\nonumber\\
	&=\frac{\pi^{\frac{3}{2}}(k-1)^{s-\frac12}z^{\frac{s-1}{2k-2}}}{\cos\left(\frac{\pi s}{2}\right)} \left(1+
	\sum_{\mu=1}^{M}
	\,
	\frac{(s)_{2\mu}(1-s)_{2\mu(k-1)}}{2^{2\mu k}\,
		(k-1)^{2\mu(k-1)} (\frac12)_{\mu}\mu!(-z e^{i\pi(k-2)})^{\mu}}+O_k\left( |z|^{-(M+1)}\right) \right)\nonumber\\
	& \quad-\frac{\pi^{\frac32}(k-1)^{s-k+\frac12}z^{\frac{s-k}{2k-2}}\Gamma(s+1)}{  2^{k-1}\Gamma(1-k+s)\sin \left(\tfrac{(k-s)\pi}{2}\right)}  \left(
	\sum_{\mu=0}^{M}
	\,
	\frac{(s+1)_{2\mu}\,
		(k - s)_{2\mu(k-1)}}{2^{2\mu k}\,(k-1)^{2\mu(k-1)} (\frac32)_{\mu}\mu!(-z e^{i\pi(k-2)})^{\mu}}
	+O_k\left( |z|^{-(M+1)} \right)
	\right).
\end{align*}
Hence, with $z=\left(\frac{y}{2k-2}\right)^{2k-2}$, where $y=2\pi n^{\frac{1}{k-1}}$, so that $z=\left(\frac{\pi}{k-1}\right)^{2k-2}n^2$, we have
\begin{align}\label{meijer G asymptotic}
	&G_{2, \, \,\,\,\,\,2k}^{k+1, \, \, 2}\left( \begin{matrix}
		\left\{\frac{2k-3+s}{2k-2}, \frac{k-2+s}{2k-2}\right\},\left\{\right\}\\
		\left\{0,\frac{1}{k-1},\cdots,\frac{k-2}{k-1}, \frac{sk-1}{2k-2}, \frac{sk-1}{2k-2}+\frac{1}{2}\right\}, \left\{\frac{1}{2k-2},\frac{3}{2k-2},\cdots,\frac{2k-3}{2k-2}\right\}
	\end{matrix} \bigg | z\right)-\frac{\pi^{\frac{3}{2}}(k-1)^{s-\frac12}z^{\frac{s-1}{2k-2}}}{\cos\left(\frac{\pi s}{2}\right)}\nonumber\\
&=	\sum_{\mu=1}^{M}A_{s,k}(\mu)n^{\frac{s-1}{k-1}-2\mu}+O_k\left(n^{\frac{\sigma-1}{k-1}-2(M+1)}\right)+\sum_{\mu=0}^{M}B_{s,k}(\mu)n^{\frac{s-k}{k-1}-2\mu}+O_k\left(n^{\frac{\sigma-k}{k-1}-2(M+1)}\right),
\end{align}
where
\begin{align*}
A_{s,k}(\mu)&:=\frac{\pi^{s+\frac{1}{2}-(2k-2)\mu}\sqrt{k-1}(s)_{2\mu}(1-s)_{2\mu(k-1)}}{2^{2\mu k}\,
	 (\frac12)_{\mu}\mu!(- e^{i\pi(k-2)})^{\mu}\cos\left(\frac{\pi s}{2}\right)},\nonumber\\
B_{s,k}(\mu)&:=	 \frac{\pi^{\frac{3}{2}+s-k-2\mu(k-1)}\sqrt{k-1}\Gamma(s+1+2\mu)\,
	(k - s)_{2\mu(k-1)}}{2^{2\mu k+k-1} (\frac32)_{\mu}\mu!(-e^{i\pi(k-2)})^{\mu}\Gamma(1-k+s)\sin \left(\tfrac{(s-k)\pi}{2}\right)}.
\end{align*}
We now have the necessary ingredients to prove the analytic extension of Theorem \ref{cs type formula}. 

\begin{proof}[Theorem \textup{\ref{analytically extended}}][]
	Invoking Theorems \ref{cs type formula} and \ref{integral conjecture}, we deduce that for $1<\textup{Re}(s)<3$,
\small\begin{align}\label{almost extended}
	\mathcal{L}_k(s)&=\zeta^2(s)+\frac{1}{(k-1)\Gamma(s)}\Gamma\left(\frac{1-s}{k-1}\right)\Gamma\left(\frac{ks-1}{k-1}\right)\zeta\left(\frac{ks-1}{k-1}\right)+\frac{2^{s-1}}{\sqrt{\pi(k-1)}\Gamma^2(s)}\nonumber\\
	&\times\sum_{n=1}^{\infty} \frac{\sigma^{(k-1)}_{ks-1}(n)}{n^{\frac{ks-1}{k-1}}}\bigg\{ \mathcal{G}_k\left(s,2\pi n^{\frac{1}{k-1}}\right)-\frac{\pi^{s+\frac{1}{2}}\sqrt{k-1}}{\cos\left(\frac{\pi s}{2}\right)}n^{\frac{s-1}{k-1}}-\sum_{\mu=1}^{M}A_{s,k}(\mu)n^{\frac{s-1}{k-1}-2\mu}-\sum_{\mu=0}^{M}B_{s,k}(\mu)n^{\frac{s-k}{k-1}-2\mu}\bigg\}\nonumber\\
	&+\sum_{n=1}^{\infty} \frac{\sigma^{(k-1)}_{ks-1}(n)}{n^{\frac{ks-1}{k-1}}}\left(\sum_{\mu=1}^{M}A_{s,k}(\mu)n^{\frac{s-1}{k-1}-2\mu}+\sum_{\mu=0}^{M}B_{s,k}(\mu)n^{\frac{s-k}{k-1}-2\mu}\right).
\end{align}
\normalsize
Now in the aforementioned vertical strip, we have, using \eqref{gen_sigma_z},
\begin{align}\label{converted to zeta}
	&\sum_{n=1}^{\infty} \frac{\sigma^{(k-1)}_{ks-1}(n)}{n^{\frac{ks-1}{k-1}}}\left(\sum_{\mu=1}^{M}A_{s,k}(\mu)n^{\frac{s-1}{k-1}-2\mu}+\sum_{\mu=0}^{M}B_{s,k}(\mu)n^{\frac{s-k}{k-1}-2\mu}\right)\nonumber\\
	&=\sum_{\mu=1}^{M}A_{s,k}(\mu)\zeta(s+2\mu)\zeta(1-s+2\mu(k-1))+\sum_{\mu=0}^{M}B_{s,k}(\mu)\zeta(s+1+2\mu)\zeta(k-s+2\mu(k-1)).
\end{align}
Combining \eqref{almost extended} and \eqref{converted to zeta}, we see that \eqref{analytically extended eqn} holds for $1<\textup{Re}(s)<3$. However, \eqref{meijer G asymptotic} and \eqref{gen_sigma_z} imply that that the series involving the Meijer $G$-function in \eqref{analytically extended eqn} converges absolutely and uniformly in $-2M-1<\textup{Re}(s)<(2M+2)(k-1)$ except\footnote{We need to exclude these values, for, the Meijer $G$-function ceases to exist at these values owing to the fact that the condition \cite[p.~143]{Luke} that $a_{\ell_1}-b_{\ell_2}\notin\mathbb{Z}^{+}$ is not satisfied for every $\ell_1=1,2,\cdots,n$ and $\ell_2=1,2,\cdots,m$.} for $s=1+j(k-1), j\in\mathbb{N}\cup \{0\}$ and $j<2M+2-\frac{1}{k-1}$, and is, therefore, analytic in this region. Also, the sum of the remaining expressions on the right-hand side of \eqref{almost extended} is analytic in the same region except for a double pole at $s=1$ and additional simple poles, due to $\Gamma\left(\frac{ks-1}{k-1}\right)$, at $s=1/k$ and  $\frac1k-\frac{(k-1)(2n+1)}{k}$ for $0\leq n <\frac{Mk+1}{k-1}$. 

Observe that the apparent simple poles at $s=1-2\mu$ (or $-2\mu$) arising from $\zeta(s+2\mu)$ (or $\zeta(s+1+2\mu)$) in the sums over $\mu$ involving zeta functions in \eqref{analytically extended eqn} get annihilated by the double poles of $\Gamma^{2}(s)$ at the same values. Also, the simple poles of $\zeta(1-s+2\mu(k-1))$ (or $\zeta(k-s+2\mu(k-1))$) get annihilated by the simple poles of $\Gamma(1-s+2\mu(k-1))$ (or $\Gamma(k-s+2\mu(k-1))$) occurring in $A_{s,k}(\mu)$ (or $B_{s,k}(\mu)$) respectively. Finally, we leave it to the reader to verify that the apparent poles of $\Gamma\left(\frac{1-s}{k-1}\right)$ at $s=1+j(k-1)$, $j\in\mathbb{N}\cup \{0\}$ and $j<2M+2-\frac{1}{k-1}$, are annihilated by the poles at the same values of the series over Meijer $G$-function and of the sums over $\mu$.

Thus, the right-hand side of \eqref{analytically extended eqn} gives another way of meromorphically continuing $\mathcal{L}_k(s)$ in the whole $s$-complex plane as $M$ can be made arbitrarily large, and with the poles, precisely, at the same values of $s$ as those given in Theorem \ref{ac}.

On the other hand, Theorem \ref{ac},  the aforementioned discussion and the principle of analytic continuation imply that the identity in \eqref{analytically extended eqn} holds in the wider region $-2M-1<\textup{Re}(s)<(2M+2)(k-1)$ for any $M\in\mathbb{N}\cup\{0\}$ than just in $1<\textup{Re}(s)<3$.
\end{proof}

We now give an application of Theorem \ref{cs type formula} which expresses $\mathcal{F}_3(\frac{3}{2})$ in terms of an infinite series involving Bessel and modified  Bessel functions of the second kind.
\begin{proof}[Corollary \textup{\ref{k=3 s=3/2}}][]

Let $k=3$ and $s=3/2$ in Theorem \ref{cs type formula} and employ Theorem \ref{integral conjecture} with $y=2\pi\sqrt{n}$. This gives
	\begin{align}\label{k=3 s=3/2 eqn1}
	%\sum_{\substack{m, n=1}}^{\infty}\frac{1}{(m^3+mn)^\frac{3}{2}}
	\sum_{n=1}^{\infty}\frac{d^{\left(\frac{1}{3}\right)}(n)}{n^{3/2}}&=\zeta\left(\frac{9}{2}\right)+\zeta^2\left(\frac{3}{2}\right)+\frac{1}{\sqrt{\pi}}\Gamma\left(\frac{-1}{4}\right)\Gamma\left(\frac{7}{4}\right)\zeta\left(\frac{7}{4}\right)\nonumber\\	
	&\quad+\sum_{n=1}^{\infty} \frac{\sigma^{(2)}_{7/2}(n)}{n^{\frac{7}{4}}}\left(\frac{4}{\pi^{3/2}}G_{2, 6}^{4, 2}\left( \begin{matrix}
		\left\{\frac{9}{8}, \frac{5}{8}\right\},\left\{\right\}\\
		\left\{0,\frac{1}{2}, \frac{7}{8}, \frac{11}{8}\right\}, \left\{\frac{1}{4},\frac{3}{4}\right\}
	\end{matrix} \bigg | \tfrac{\pi^4n^2}{16} \right)+8\sqrt{\pi} n^{1/4}\right).
\end{align}
From \cite[p.~668, formula \textbf{8.4.23.22}]{Prudnikov},
\begin{align*}
I_{\mu}\left(\frac{1}{\sqrt{x}}\right)K_{\nu}\left(\frac{1}{\sqrt{x}}\right)=\frac{1}{2\sqrt{\pi}}G^{2, 2}_{4, 2}\left( \begin{matrix}
	\left\{1-\frac{\mu+\nu}{2}, 1+\frac{\nu-\mu}{2}\right\},\left\{1+\frac{\mu-\nu}{2}, 1+\frac{\mu+\nu}{2}\right\}\\
	\left\{\frac{1}{2},1\right\}, \left\{\right\}
\end{matrix} \bigg | x\right).
\end{align*}
Using the above identity twice, once with $\mu=\nu=1/4$, and then, with $\mu=\nu=-3/4$, we find upon simplification that
\begin{align*}
&-4\sqrt{\pi}\left(\pi\sqrt{n}\left(I_{\frac{1}{4}}(\pi\sqrt{n})K_{\frac{1}{4}}(\pi\sqrt{n})+3I_{-\frac{3}{4}}(\pi\sqrt{n})K_{\frac{3}{4}}(\pi\sqrt{n})\right)-2\right)\nonumber\\
&=-2\pi\sqrt{n}\left\{G^{2, 2}_{4, 2}\left( \begin{matrix}
	\left\{\frac{3}{4}, 1\right\},\left\{1, \frac{5}{4}\right\}\\
	\left\{\frac{1}{2},1\right\}, \left\{\right\}
\end{matrix} \bigg | \frac{1}{n\pi^2}\right)+3G^{2, 2}_{4, 2}\left( \begin{matrix}
\left\{\frac{7}{4}, 1\right\},\left\{1, \frac{1}{4}\right\}\\
\left\{\frac{1}{2},1\right\}, \left\{\right\}
\end{matrix} \bigg | \frac{1}{n\pi^2}\right)\right\}+8\sqrt{\pi}.
\end{align*}
Thus, from \eqref{k=3 s=3/2 eqn}, \eqref{k=3 s=3/2 eqn1} and the above equation, we will be done if we show that
\begin{align}\label{meijer g to be proved}
	G^{2, 2}_{4, 2}\left( \begin{matrix}
	\left\{\frac{3}{4}, 1\right\},\left\{1, \frac{5}{4}\right\}\\
	\left\{\frac{1}{2},1\right\}, \left\{\right\}
\end{matrix} \bigg | \frac{1}{n\pi^2}\right)+3G^{2, 2}_{4, 2}\left( \begin{matrix}
	\left\{\frac{7}{4}, 1\right\},\left\{1, \frac{1}{4}\right\}\\
	\left\{\frac{1}{2},1\right\}, \left\{\right\}
\end{matrix} \bigg | \frac{1}{n\pi^2}\right)
=\frac{-2}{n^{3/4}\pi^{5/2}}G_{2, 6}^{4, 2}\left( \begin{matrix}
	\left\{\frac{9}{8}, \frac{5}{8}\right\},\left\{\right\}\\
	\left\{0,\frac{1}{2}, \frac{7}{8}, \frac{11}{8}\right\}, \left\{\frac{1}{4},\frac{3}{4}\right\}
\end{matrix} \bigg | \tfrac{\pi^4n^2}{16} \right).
\end{align}
We now prove this identity by employing various properties of the Meijer $G$-function.

Using the identity \cite[p.~416, formula \textbf{16.19.1}]{nist}
\begin{align*}
G^{m, n}_{p, q}\left( \begin{matrix}
	\left\{a_1,\cdots,a_n\right\},\left\{a_{n+1}\cdots,a_p\right\}\\
	\left\{b_1,\cdots,b_m\right\}, \left\{b_{m+1},\cdots,b_q\right\}
\end{matrix} \bigg | x\right)=G^{n, m}_{q, p}\left( \begin{matrix}
\left\{1-b_1,\cdots,1-b_m\right\},\left\{1-b_{m+1}\cdots,1-b_q\right\}\\
\left\{1-a_1,\cdots,1-a_n\right\}, \left\{1-a_{n+1},\cdots,1-a_p\right\}
\end{matrix} \bigg | \frac{1}{x}\right)
\end{align*}
for each of the two Meijer $G$-functions on the left-hand side of \eqref{meijer g to be proved}, and \cite[p.~416, formula \textbf{16.19.4}]{nist} 
\begin{align*}
	&G^{m, n}_{p, q}\left( \begin{matrix}
		\left\{a_1,\cdots,a_n\right\},\left\{a_{n+1}\cdots,a_p\right\}\\
		\left\{b_1,\cdots,b_m\right\}, \left\{b_{m+1},\cdots,b_q\right\}
	\end{matrix} \bigg | x\right)\nonumber\\
&=\frac{2^{p+1+b_1+\cdots+b_q-m-n-a_1-\cdots-a_p}}{\pi^{m+n-\frac{1}{2}(p+q)}}	G^{2m, 2n}_{2p, 2q}\left( \begin{matrix}
	\left\{\frac{a_1}{2},\frac{1+a_1}{2},\cdots,\frac{a_n}{2}, \frac{1+a_n}{2}\right\},\left\{\frac{a_{n+1}}{2},\frac{1+a_{n+1}}{2}\cdots,\frac{a_p}{2},\frac{1+a_p}{2}\right\}\\
	\left\{\frac{b_1}{2},\frac{1+b_1}{2},\cdots,\frac{b_m}{2},\frac{1+b_m}{2}\right\}, \left\{\frac{b_{m+1}}{2}, \frac{1+b_{m+1}}{2}\cdots,\frac{b_q}{2}.\frac{1+b_q}{2}\right\}
\end{matrix} \bigg | \frac{x^2}{2^{2q-2p}}\right),
\end{align*}
for the one on the right-hand side, we see that this is equivalent to proving
\begin{align*}
	G^{2, 2}_{2, 4}\left( \begin{matrix}
		\left\{\frac{1}{2}, 0\right\},\left\{\right\}\\
		\left\{\frac{1}{4}, 0\right\}, \left\{0, -\frac{1}{4}\right\}
	\end{matrix} \bigg | n\pi^2\right)+3G^{2, 2}_{2, 4}\left( \begin{matrix}
		\left\{\frac{1}{2}, 0\right\},\left\{\right\}\\
		 \left\{-\frac{3}{4},0\right\}, \left\{\frac{3}{4},0\right\},
	\end{matrix} \bigg | n\pi^2\right)
	=\frac{-2}{n^{3/4}\pi^{3/2}}G_{1, 3}^{2, 1}\left( \begin{matrix}
		\left\{\frac{5}{4}\right\},\left\{\right\}\\
		\left\{0,\frac{7}{4}\right\}, \left\{\frac{1}{2}\right\}
	\end{matrix} \bigg | n\pi^2 \right),
\end{align*}
which, in turn, using \cite[p.~416, formula \textbf{16.19.3}]{nist} 
\begin{align*}
	G^{m, n}_{p, q}\left( \begin{matrix}
		\left\{a_0,a_1,\cdots,a_n\right\},\left\{a_{n+1}\cdots,a_p\right\}\\
		\left\{b_1,\cdots,b_m\right\}, \left\{b_{m+1},\cdots,b_q,a_0\right\}
	\end{matrix} \bigg | x\right)=G^{m, n-1}_{p-1, q-1}\left( \begin{matrix}
	\left\{a_1,\cdots,a_n\right\},\left\{a_{n+1}\cdots,a_p\right\}\\
	\left\{b_1,\cdots,b_m\right\}, \left\{b_{m+1},\cdots,b_q\right\}
\end{matrix} \bigg | x\right)
\end{align*}
for the functions on the left-hand side of \eqref{meijer g to be proved1} and \cite[p.~416, formula \textbf{16.19.2}]{nist} 
\begin{align*}
	G^{m, n}_{p, q}\left( \begin{matrix}
		\left\{a_0,a_1,\cdots,a_n\right\},\left\{a_{n+1}\cdots,a_p\right\}\\
		\left\{b_1,\cdots,b_m\right\}, \left\{b_{m+1},\cdots,b_q,a_0\right\}
	\end{matrix} \bigg | x\right)=x^{-\mu}G^{m, n}_{p, q}\left( \begin{matrix}
	\left\{,a_1+\mu,\cdots,a_n+\mu\right\},\left\{a_{n+1}+\mu\cdots,a_p+\mu\right\}\\
	\left\{b_1+\mu,\cdots,b_m+\mu\right\}, \left\{b_{m+1}+\mu,\cdots,b_q+\mu\right\}
\end{matrix} \bigg | x\right),
\end{align*}
with $\mu=3/4$ for the one on the right, is equivalent to proving
\begin{align}\label{meijer g to be proved1}
	G^{2, 1}_{1,3}\left( \begin{matrix}
		\left\{\frac{1}{2}\right\},\left\{\right\}\\
		\left\{\frac{1}{4}, 0\right\}, \left\{-\frac{1}{4}\right\}
	\end{matrix} \bigg | n\pi^2\right)+3G^{2, 1}_{1, 3}\left( \begin{matrix}
		\left\{\frac{1}{2}\right\},\left\{\right\}\\
		\left\{-\frac{3}{4},0\right\}, \left\{\frac{3}{4}\right\},
	\end{matrix} \bigg | n\pi^2\right)
	=-2G^{2, 1}_{1,3}\left( \begin{matrix}
		\left\{\frac{1}{2}\right\},\left\{\right\}\\
		\left\{\frac{-3}{4}, 1\right\}, \left\{-\frac{1}{4}\right\}
	\end{matrix} \bigg | n\pi^2\right).
	\end{align}
Now subtracting both sides of Equation (1.3.16) of \cite[p.~9]{mathai saxena} from the corresponding sides of Equation (1.3.15) and simplifying, we get for $1\leq m<q$,
\begin{align}\label{b1 minus bq}
&G^{m, n}_{p, q}\left( \begin{matrix}
	\left\{a_1,\cdots,a_n\right\},\left\{a_{n+1}\cdots,a_p\right\}\\
	\left\{b_1+1,\cdots,b_m\right\}, \left\{b_{m+1},\cdots,b_q\right\}
\end{matrix} \bigg | x\right)+G^{m, n}_{p, q}\left( \begin{matrix}
	\left\{a_1,\cdots,a_n\right\},\left\{a_{n+1}\cdots,a_p\right\}\\
	\left\{b_1,\cdots,b_m\right\}, \left\{b_{m+1},\cdots,b_q+1\right\}
\end{matrix} \bigg | x\right)\nonumber\\
&=(b_1-b_q)G^{m, n}_{p, q}\left( \begin{matrix}
	\left\{a_1,\cdots,a_n\right\},\left\{a_{n+1}\cdots,a_p\right\}\\
	\left\{b_1,\cdots,b_m\right\}, \left\{b_{m+1},\cdots,b_q\right\}
\end{matrix} \bigg | x\right).
\end{align}
Using \eqref{b1 minus bq} with $m=2, n=p=1, q=3$, $a_1=1/2, b_1=-3/4, b_q=-1/4$, and $x=n\pi^2$, we obtain
\begin{align}\label{meijer g to be proved2}
	G^{2, 1}_{1,3}\left( \begin{matrix}
		\left\{\frac{1}{2}\right\},\left\{\right\}\\
		\left\{\frac{1}{4}, 0\right\}, \left\{-\frac{1}{4}\right\}
	\end{matrix} \bigg | n\pi^2\right)+G^{2, 1}_{1, 3}\left( \begin{matrix}
		\left\{\frac{1}{2}\right\},\left\{\right\}\\
		\left\{-\frac{3}{4},0\right\}, \left\{\frac{3}{4}\right\},
	\end{matrix} \bigg | n\pi^2\right)
	=-\frac{1}{2}G^{2, 1}_{1,3}\left( \begin{matrix}
		\left\{\frac{1}{2}\right\},\left\{\right\}\\
		\left\{\frac{-3}{4}, 0\right\}, \left\{-\frac{1}{4}\right\}
	\end{matrix} \bigg | n\pi^2\right).
\end{align}
Equations \eqref{meijer g to be proved1} and \eqref{meijer g to be proved2} imply that we will be done provided it is shown that
\begin{align*}
G^{2, 1}_{1,3}\left( \begin{matrix}
\left\{\frac{1}{2}\right\},\left\{\right\}\\
\left\{\frac{-3}{4}, 1\right\}, \left\{-\frac{1}{4}\right\}
\end{matrix} \bigg | n\pi^2\right)+G^{2, 1}_{1, 3}\left( \begin{matrix}
\left\{\frac{1}{2}\right\},\left\{\right\}\\
\left\{-\frac{3}{4},0\right\}, \left\{\frac{3}{4}\right\}
\end{matrix} \bigg | n\pi^2\right)=\frac{1}{4}G^{2, 1}_{1,3}\left( \begin{matrix}
\left\{\frac{1}{2}\right\},\left\{\right\}\\
\left\{\frac{-3}{4}, 0\right\}, \left\{-\frac{1}{4}\right\}
\end{matrix} \bigg | n\pi^2\right).
\end{align*}
This follows from another application of \eqref{b1 minus bq} with $a_1=1/2, b_1=0$ and $b_q=-1/4$ This establishes the identity. 
\end{proof}
Next, we derive an identity between two infinite series involving the generalized divisor functions using two of the three different representations for $\mathcal{F}_k(s)$ derived in the paper. 
\begin{proof}[Corollary \textup{\ref{limiting}}][]
Equate the right-hand sides of \eqref{formula3} and \eqref{lks final other} and then let $s\to1/k$ while observing that
\begin{align*}
\lim_{s\to\frac{1}{k}}\left[\frac{1}{2}\zeta(ks)+\frac{1}{(k-1)\Gamma(s)}\Gamma\left(\frac{1-s}{k-1}\right)\Gamma\left(\frac{ks-1}{k-1}\right)\zeta\left(\frac{ks-1}{k-1}\right)\right]=\frac{1}{2(k-1)}\left(\gamma k-\log(2\pi)+\psi\left(\frac{1}{k}\right)\right)
\end{align*}
and
\begin{align*}
\int_1^{\infty} \cos(2\pi n x) x^{-\frac{1}{k}-2} \, dx=\frac{k}{k+1}\left\{\pFq{1}{2}{-\frac{1}{2}-\frac{1}{2k}}{\frac{1}{2},\frac{1}{2}-\frac{1}{2k}}{-n^2\pi^2}+(2\pi n)^{1+\frac{1}{k}}\Gamma\left(-\frac{1}{k}\right)\sin\left(\frac{\pi}{2k}\right)\right\},
\end{align*}
to arrive at \eqref{general limiting case} upon simplification.
\end{proof}

\begin{proof}[Corollary \textup{\ref{limiting k=3}}][]
This follows from Corollary \ref{limiting} upon letting $k=3$ and using \eqref{def2varbessel} and Slater's theorem to reduce the Meijer $G$-function in terms of ${}_1F_{2}$-hypergeometric function.	
\end{proof}

\section{A Bessel series representation for the Dirichlet series of $d^{\left(\frac{1}{2}\right)}(n)$}
We now concentrate on the analogue of Theorem \ref{cs type formula} when $k=2$. Observe that one cannot let $k=2$ in Theorem \ref{cs type formula} itself since the result there holds for $1<\textup{Re}(s)<k-1$. Hence, we assume $1<\textup{Re}(s)<2$, and then work out the corresponding result.
\begin{proof}[Theorem \textup{\ref{k=2 cs type formula}}][]
We omit the details since it is along the similar lines to that of Theorem \ref{cs type formula} except that we have to consider the contribution of the residue at an extra simple pole $-s-1$, besides the simple poles at $-s$ and $\frac{1-ks}{k-1}=1-2s$, owing to the condition $1<\textup{Re}(s)<2$. The residue at this pole is  $R_{-s-1} = -\pi (2\pi)^{-2s} s \, \zeta(s-1)\zeta(s+1)$. One then uses Theorem \ref{integral conjecture} with $y=2\pi n$ and employs the well-known fact \cite{wolfram} with $a=s/2, b=s-1/2$ to arrive at
\begin{align*}
G^{3, 2}_{2, 4}\left( \begin{matrix}
	\left\{\frac{1+s}{2}, \frac{s}{2}\right\},\left\{\right\}\\
	\left\{0,s-\frac{1}{2}, s\right\}, \left\{\frac{1}{2}\right\},
\end{matrix} \bigg | \pi^2n^2\right)=(-1)^{n+1}\frac{\sqrt{2}\pi^{s+\frac{3}{2}}n^{s-\frac{1}{2}}}{\sin(\pi s)}Y_{s-\frac{1}{2}}(\pi n).
\end{align*}
\end{proof}
%\begin{remark}
%One may also use the evaluation of the integral in Theorem \ref{integral conjecture} for $k=2$ given in \cite[p.~68]{atk}. 
%\end{remark}
\begin{proof}[Corollary \textup{\ref{similar to atkinson's}}][]
	The proof follows simply by equating the two representations for $\mathcal{L}_2(s)$ from \eqref{by atkinson} and Theorem \ref{k=2 cs type formula}. 	
\end{proof}

\begin{proof}[Corollary \textup{\ref{pc}}][]
The first result in this corollary follows by letting $s\to1/2$ in Corollary \ref{similar to atkinson's} and then making use of the limit evaluation
\begin{align*}
\lim_{s\to\frac{1}{2}}\left(\frac{1}{2}\zeta(2s)+\frac{\Gamma(1-s)\Gamma\left(2s-1\right)\zeta(2s-1)}{\Gamma(s)}\right)=\frac{\gamma}{2}-\frac{1}{2}\log(2\pi)-\log(2),
\end{align*}
 which can be easily proved by using the Laurent series expansion of the Gamma and Riemann zeta functions. The second one follows by letting $s=3/2$ in Corollary \ref{similar to atkinson's}.
	
\end{proof}
\section{Concluding remarks}\label{cr}

While Wigert's opinion that his divisor function $d^{(\frac{1}{k})}(n)$ is not easy to work with is correct, we hope to have convinced the reader it is definitely worth studying. We list below some problems for future work.

(1) Matsumoto \cite{matsumoto sugaku} (see also \cite[Equation (1.5)]{choiematsumoto} showed that $\mathcal{L}_2(s)$ satisfies the functional equation
\begin{align*}
\mathcal{L}_2(s)&=\frac{\Gamma(1-s_1)\Gamma(s_1+s_2-1)}{\Gamma(s_2)}	\zeta(s_1+s_2-1)\\
&\quad+\Gamma(1-s_1)\left\{F_+(1-s_2, 1-s_1; \mathfrak{U}_1)+F_-(1-s_2, 1-s_1; \mathfrak{U}_1)\right\},
\end{align*}
where
\begin{equation*}
F_\pm(s_1, s_2;\mathfrak{U}_1):=\sum_{k\geq 1}\sigma_{s_1+s_2-1}(k)\Psi(s_2, s_1+s_2;\pm2\pi ik),
\end{equation*}
with
\begin{equation*}
\Psi(a,b;x):=\frac{1}{\Gamma(a)}\int_{0}^{e^{i\phi}\infty}e^{-xy}y^{a-1}(1+y)^{b-a-1}\, dy
\end{equation*}
for Re$(a)>0, -\pi<\phi<\pi$, and $|\phi+\arg(x)|<\pi/2$ is the confluent hypergeometric function. 
In light of the above, it would certainly of merit to investigate whether, for $k>2$, $\mathcal{L}_k(s)$ also admits a functional equation.

(2) The question on the distribution of zeros of $\mathcal{L}_2(s)$ was first proposed by Zhao \cite{zhao} and has been studied by Nakamura and Pa\'{n}kowski \cite{nakamura pankowski} Matsumoto and Sh\=oji \cite{matsumoto moscow}. The latter two authors showed that the distribution of the zeros of $\mathcal{L}_2(s)$ is similar in some sense to that of the zeros of Hurwitz zeta function.
 In the similar vein, what can be said about the zeros of $\mathcal{L}_k(s)$ for $k>2$?
 
(3) Lastly, it would be interesting to obtain a Vorono\"{\dotlessi} summation formula for $d^{(\frac{1}{k})}(n)$. 
%Vorono\"{\dotlessi} summation formula for $d^{\left(\frac{1}{k}\right)}(n)$? Can Theorem \ref{cs type formula} be used for that purpose? Give the idea.

\section*{Acknowledgements}
The first author is supported by the Core Research Grant CRG/2023/002698 and MATRICS grant MTR/2023/000837 of ANRF. The second author is supported by the Swarnajayanti Fellowship grant SB/SJF/2021-22/08 of ANRF (Government of India) and by the N Rama Rao chair professorship at IIT Gandhinagar.

\section*{Appendix}
\begin{center}
	\vspace{1mm}
	\textbf{A reduction of a Meijer $G$-function to  a sum of two infinite series  }\\
	\end{center}
	\vspace{1mm}
\begin{theorem}\label{meijerg specialization}
Let $\mathcal{G}_k(s,y)$ be defined in \eqref{short notation}.	For $y>0, k\in\mathbb{N}, k\geq2$ and $s\in\mathbb{C}$ such that the difference between $\frac{2k-3+s}{2k-2}$ or $\frac{k-2+s}{2k-2}$ and each of $0,\frac{1}{k-1},\cdots,\frac{k-2}{k-1}, \frac{sk-1}{2k-2}$, or $ \frac{sk-1}{2k-2}+\frac{1}{2}$ is not a positive integer, we have
	\begin{align}\label{meijerg specialization eqn}
		\frac{2^{s-2}}{\sqrt{\pi(k-1)}\Gamma(s)}\mathcal{G}_k(s,y)
		&=\frac{y^{ks-1}}{\Gamma(s)}\bigg\{\sum_{j=0}^{\infty}\frac{(-y^{k-1})^{j}}{j!}\Gamma(j+s)\Gamma(1+j-k(s+j))\sin\left(\frac{\pi}{2}(ks+(k-1)j)\right)\nonumber\\
		&\quad+\frac{1}{(k-1)}\sum_{m=0}^{\infty}\frac{(-1)^m}{(2m)!}y^{1+2m-ks}\Gamma\left(\frac{1+2m-s}{k-1}\right)\Gamma\left(\frac{ks-2m-1}{k-1}\right)\bigg\}.
	\end{align}
\end{theorem}
\begin{proof}
	By Slater's theorem \cite[p.~145]{Luke}, 
	\begin{align*}
		%&G_{2, \, \, 2k}^{k+1, \, \, 2}\left( \begin{matrix}
		%	\left\{\frac{2k-3+s}{2k-2}, \frac{k-2+s}{2k-2}\right\}\\
		%		\left\{0,\frac{1}{k-1},\cdots,\frac{k-2}{k-1}, \frac{sk-1}{2k-2}, \frac{sk-1}{2k-2}+\frac{1}{2}\right\}, \left\{\frac{1}{2k-2},\frac{3}{2k-2},\cdots,\frac{2k-3}{2k-2}\right\}
		%\end{matrix} \bigg | \left(\frac{y}{2k-2}\right)^{2k-2} \right)\nonumber\\
		\mathcal{G}_k(s,y)&=\sum_{h=1}^{k+1}\frac{\prod_{j=1}^{k+1}\Gamma(b_j-b_h)^*\Gamma\left(b_h+\frac{1-s}{2k-2}\right) \Gamma\left(\frac{1}{2}+b_h+\frac{1-s}{2k-2}\right)}{\prod_{j=k+2}^{2k}\Gamma(1+b_h-b_j)}\left(\frac{y}{2k-2}\right)^{(2k-2)b_h}\nonumber\\
		&\quad\qquad\times\pFq2{2k-1}{b_h+\frac{1-s}{2k-2}, \frac{1}{2}+b_h+\frac{1-s}{2k-2}}{\left\langle 1+b_h-b_{i}\right\rangle_{i=1\atop i\neq h}^{2k}}{(-1)^{k+1}\left(\frac{y}{2k-2}\right)^{2k-2}}.
	\end{align*}
	We prove only the case when $k$ is odd, that is, $k=2\ell+1$. The corresponding case when $k$ is even can be similarly proved. 
	
	To that end, separating the terms corresponding to $h=2\ell+1$ and $h=2\ell+2$, noting that 
	\begin{align}\label{bi}
		b_h&=\frac{h-1}{2\ell}\hspace{5mm} (1\leq h\leq2\ell),\nonumber\\
		b_{2\ell+1}&=\frac{(2\ell+1)s-1}{4\ell},\hspace{3mm} b_{2\ell+2}=\frac{2\ell(s+1)+s-1}{4\ell},
	\end{align}
	and using the duplication formula for the Gamma function, we see that
	\begin{align}\label{meijerG simplification}
		\mathcal{G}_{2\ell+1}(s,y)&=\sum_{h=1}^{2\ell}\frac{\prod_{j=1\atop j\neq h}^{2\ell}\Gamma\left(\frac{j-h}{2\ell}\right)\Gamma\left(\frac{2\ell s+s-1}{4\ell}-\frac{h-1}{2\ell}\right))\Gamma\left(\frac{1}{2}+\frac{2\ell s+s-1}{4\ell}-\frac{h-1}{2\ell}\right)}{\prod_{j=1}^{2\ell}\Gamma\left(1+\frac{2h-2j-1}{4\ell}\right)}\frac{\sqrt{\pi}}{2^{\frac{(h-1)}{\ell}+\frac{1-s}{2\ell}-1}}\nonumber\\
		&\quad\times\Gamma\left(\frac{h-1}{\ell}+\frac{1-s}{2\ell}\right)\left(\frac{y}{4\ell}\right)^{2h-2}\pFq2{4\ell+1}{\frac{h-1}{2\ell}+\frac{1-s}{4\ell},\frac{1}{2}+\frac{h-1}{2\ell}+\frac{1-s}{4\ell}}{\left\langle1+\frac{h-1}{2\ell}-b_{i}\right\rangle_{i=1\atop i\neq h}^{4\ell+2}}{\left(\frac{y}{4\ell}\right)^{4\ell}}\nonumber\\
		&\quad+\frac{\prod_{j=1}^{2\ell}\Gamma\left(\frac{j-1}{2\ell}-\frac{2\ell s+s-1}{4\ell}\right)\Gamma\left(\frac{1}{2}\right)\frac{\sqrt{\pi}}{2^{s-1}}\Gamma(s)\left(\frac{y}{4\ell}\right)^{s(2\ell+1)-1}}{\prod_{j=1}^{2\ell}\Gamma\left(1+\frac{s(2\ell+1)-2j}{4\ell}\right)}\nonumber\\
		&\quad\times\pFq2{4\ell+1}{\frac{s}{2},\frac{s+1}{2}}{\left\langle1+\frac{s(2\ell+1)-1}{4\ell}-b_{i}\right\rangle_{i=1\atop i\neq2\ell+1}^{4\ell+2}}{\left(\frac{y}{4\ell}\right)^{4\ell}}\nonumber\\
		&\quad+\frac{\prod_{j=1}^{2\ell}\Gamma\left(\frac{j-1}{2\ell}-\frac{2\ell (s+1)+s-1}{4\ell}\right)\Gamma\left(-\frac{1}{2}\right)\frac{\sqrt{\pi}}{2^{s}}\Gamma(s+1)\left(\frac{y}{4\ell}\right)^{2\ell(s+1)+s-1}}{\prod_{j=1}^{2\ell}\Gamma\left(1+\frac{2\ell(s+1)+s-2j}{4\ell}\right)}\nonumber\\
		&\quad\times\pFq2{4\ell+1}{\frac{s+1}{2}, \frac{s}{2}+1}{\left\langle1+\frac{2\ell(s+1)+s-1}{4\ell}-b_{i}\right\rangle_{i=1\atop i\neq2\ell+2}^{4\ell+2}}{\left(\frac{y}{4\ell}\right)^{4\ell}}.
	\end{align}
%	The Gauss multiplication formula is given by
%	\begin{equation}\label{gmf}
%		\prod_{k=1}^{m}\Gamma\left(z+\frac{k-1}{m}\right)=(2\pi)^{\frac{1}{2}(m-1)}m^{\frac{1}{2}-mz}\Gamma(mz).
%	\end{equation}
	The Gamma-products occurring in \eqref{meijerG simplification} can be easily evaluated using \eqref{gmf} as given below:
	\begin{align*}
		&\prod\limits_{j=1 \atop{i\neq h}}^{2\ell}\Gamma\left({j-h \over 2\ell}\right) 
		= \frac{(-1)^{h-1}(2\pi)^{\ell-\frac{1}{2}}(2\ell)^{h-\frac{3}{2}}}{\Gamma(h)},\\
		&\prod_{j=1}^{2\ell}\Gamma\left(1+\frac{2h-2j-1}{4\ell}\right)=(2\pi)^{\ell-\frac{1}{2}}(2\ell)^{1-h}\Gamma\left(h-\frac{1}{2}\right),\\
		&\prod_{j=1}^{2\ell}\Gamma\left(\frac{j-1}{2\ell}-\frac{2\ell s+s-1}{4\ell}\right)=(2\pi)^{\ell-\frac{1}{2}}(2\ell)^{\frac{s}{2}+\ell s}\Gamma\left(\frac{1-s}{2}-\ell s\right),\\
		&\prod_{j=1}^{2\ell}\Gamma\left(1+\frac{s(2\ell+1)-2j}{4\ell}\right)=(2\pi)^{\ell-\frac{1}{2}}(2\ell)^{\frac{1-s}{2}-\ell s}\Gamma\left(\frac{s}{2}+\ell s\right),\\
		&\prod_{j=1}^{2\ell}\Gamma\left(\frac{j-1}{2\ell}-\frac{2\ell (s+1)+s-1}{4\ell}\right)=(2\pi)^{\ell-\frac{1}{2}}(2\ell)^{\frac{s}{2}+\ell s+\ell}\Gamma\left(\frac{1-s}{2}-\ell s-\ell\right),\\
		&\prod_{j=1}^{2\ell}\Gamma\left(1+\frac{2\ell(s+1)+s-2j}{4\ell}\right)=(2\pi)^{\ell-\frac{1}{2}}(2\ell)^{\frac{1-s}{2}-\ell-\ell s}\Gamma\left(\frac{s}{2}+\ell s+\ell\right).
	\end{align*}
	Employing these evaluations in \eqref{meijerG simplification}, we see that
	\begin{align}\label{meijer G simplification 1.5}
		\mathcal{G}_{2\ell+1}(s,y)=A_{\ell}( s,y)+B_{\ell}(s,y),
	\end{align}
	where
	\begin{align}\label{meijer G simplification 2}
		A_{\ell}( s,y)&:=\frac{\sqrt{\pi}2^{\frac{3}{2}-s}}{\sqrt{\ell}}\sum_{h=1}^{2\ell}(-1)^{h-1}\frac{\Gamma\left(s+\frac{s-1}{2\ell}-\frac{h-1}{\ell}\right)\Gamma\left(\frac{1-s}{2\ell}+\frac{h-1}{\ell}\right)}{\Gamma(2h-1)}y^{2h-2}\nonumber\\
		&\qquad\qquad\times\pFq2{4\ell+1}{\frac{h-1}{2\ell}+\frac{1-s}{4\ell},\frac{1}{2}+\frac{h-1}{2\ell}+\frac{1-s}{4\ell}}{\left\langle1+\frac{h-1}{2\ell}-b_{i}\right\rangle_{i=1\atop i\neq h}^{4\ell+2}}{\left(\frac{y}{4\ell}\right)^{4\ell}}\nonumber\\
		B_{\ell}(s,y)&:=\pi2^{\frac{5}{2}-2s-2\ell s}\sqrt{\ell}y^{2\ell s+s-1}\Gamma(s)\frac{\Gamma\left(\frac{1-s-2\ell s}{2}\right)}{\Gamma\left(\frac{2\ell s+s}{2}\right)}\left\{\pFq2{4\ell+1}{\frac{s}{2},\frac{s+1}{2}}{\left\langle1+\frac{s(2\ell+1)-1}{4\ell}-b_{i}\right\rangle_{i=1\atop i\neq2\ell+1}^{4\ell+2}}{\left(\frac{y}{4\ell}\right)^{4\ell}}\right.\nonumber\\
		&\quad\left.-s\left(\frac{y}{2}\right)^{2\ell}\frac{\Gamma\left(\frac{1-s-2\ell s}{2}-\ell\right)\Gamma\left(\frac{2\ell s+s}{2}\right)}{\Gamma\left(\frac{2\ell s+s}{2}+\ell\right)\Gamma\left(\frac{1-s-2\ell s}{2}\right)}\pFq2{4\ell+1}{\frac{s+1}{2},\frac{s}{2}+1}{\left\langle1+\frac{s(2\ell+1)-1}{4\ell}-b_{i}+\frac{1}{2}\right\rangle_{i=1\atop i\neq2\ell+2}^{4\ell+2}}{\left(\frac{y}{4\ell}\right)^{4\ell}}\right\}.
	\end{align}
	To simplify the right-hand side of \eqref{meijer G simplification 1.5}, we will repeatedly use the identity
	\begin{align}\label{hypergeometric identity}
		\pFq1{2\ell}{a}{c_1,c_2,\cdots,c_{\ell}}{z}&=\pFq2{4\ell+1}{\frac{a}{2},\frac{a+1}{2}}{\frac{c_1}{2},\frac{c_1+1}{2},\cdots,\frac{c_{2\ell}}{2},\frac{c_{2\ell}+1}{2},\frac{1}{2}}{\frac{z^2}{2^{4\ell}}}\nonumber\\
		&\quad+\frac{az}{c_1\cdots c_{2\ell}}\pFq2{4\ell+1}{\frac{a+1}{2},\frac{a+2}{2}}{\frac{c_1+1}{2},\frac{c_1+2}{2},\cdots,\frac{c_
				{2\ell}+1}{2}, \frac{b_\ell+2}{2},\frac{3}{2}}{\frac{z^2}{2^{4\ell}}},
	\end{align}
	which can be easily proved by using the series definition of the ${}_1F_{2\ell}$ on the left and splitting it into two sums, one corresponding to the even values of the summation index, and another corresponding to the odd values.
	
	We first simplify $B_{\ell}(y, s)$. 
	Using the values of $b_i, $ \eqref{bi}, we define
	\begin{align*}
		&\tilde{b}_i:=1+\frac{s(2\ell+1)-1}{4\ell}-\frac{i-1}{2\ell}\hspace{5mm}(1\leq i\leq 2\ell),\\
		&\tilde{b}_{2\ell+2}:=\frac{1}{2},
	\end{align*}
	and since $b_{2\ell+2+i}=\frac{2i-1}{4\ell}$ for $1\leq i\leq 2\ell$, we let
	\begin{align*}
		\tilde{b}_{2\ell+2+i}:=1+\frac{s(2\ell+1)-1}{4\ell}-\frac{2i-1}{4\ell}\hspace{5mm}(1\leq i\leq 2\ell),
	\end{align*}
	(Note that the value $\tilde{b}_{2\ell+1}$ corresponding to $b_{2\ell+1}$ is not to be considered in the two ${}_2F_{4\ell+1}$s.)
	Now for $1\leq i\leq \ell$,
	\begin{equation*}
		\tilde{b}_{\ell+i}+\frac{1}{2}=\tilde{b}_i,\hspace{4mm} \tilde{b}_{2\ell+2+\ell+i}+\frac{1}{2}=\tilde{b}_{2\ell+2+i}.
	\end{equation*}
	We let $z=(-1)^{\ell+1}\left(\frac{y}{2\ell}\right)^{2\ell}, a=s$, $c_i=2\tilde{b}_{\ell+i}$ for $1\leq i\leq \ell$ and $c_i=2\tilde{b}_{2\ell+2+i}$ for $\ell+1\leq i\leq 2\ell$ in \eqref{hypergeometric identity} and note that
	\begin{align}\label{product1}
		\frac{az}{c_1\cdots c_{2\ell}}=-s\left(\frac{y}{2}\right)^{2\ell}\frac{\Gamma\left(\frac{1-s-2\ell s}{2}-\ell\right)\Gamma\left(\frac{2\ell s+s}{2}\right)}{\Gamma\left(\frac{2\ell s+s}{2}+\ell\right)\Gamma\left(\frac{1-s-2\ell s}{2}\right)}.
	\end{align}
	The latter can be proved easily using the well-known properties of the gamma function.
	
	Using \eqref{product1} and the fact that one of the bottom parameter of the ${}_2F_{4\ell+1}$ occurring in second-last expressions on the right-hand side of \eqref{meijer G simplification 2} is $1+\frac{s(2\ell+1)-1}{4\ell}-b_{2\ell+2}=\frac{1}{2}$, and similarly, one of the bottom parameters of of the ${}_2F_{4\ell+1}$ in the corresponding last expression is $1+\frac{s(2\ell+1)-1}{4\ell}-b_{2\ell+1}+\frac{1}{2}=\frac{3}{2}$, the identity \eqref{hypergeometric identity} leads to 
	\begin{align}\label{finally3}
		%&\pi2^{\frac{5}{2}-2s-2\ell s}\sqrt{\ell}y^{2\ell s+s-1}\Gamma(s)\frac{\Gamma\left(\frac{1-s-2\ell s}{2}\right)}{\Gamma\left(\frac{2\ell s+s}{2}\right)}\left\{\pFq2{4\ell+1}{\frac{s}{2},\frac{s+1}{2}}{\left\langle1+\frac{s(2\ell+1)-1}{4\ell}-b_{i}\right\rangle_{i=1\atop i\neq2\ell+1}^{4\ell+2}}{\left(\frac{y}{4\ell}\right)^{4\ell}}\right.\nonumber\\
		%&\quad\left.-s\left(\frac{y}{2}\right)^{2\ell}\frac{\Gamma\left(\frac{1-s-2\ell s}{2}-\ell\right)\Gamma\left(\frac{2\ell s+s}{2}\right)}{\Gamma\left(\frac{2\ell s+s}{2}+\ell\right)\Gamma\left(\frac{1-s-2\ell s}{2}\right)}\pFq2{4\ell+1}{\frac{s+1}{2},\frac{s}{2}+1}{\left\langle1+\frac{s(2\ell+1)-1}{4\ell}-b_{i}+\frac{1}{2}\right\rangle_{i=1\atop i\neq2\ell+2}^{4\ell+2}}{\left(\frac{y}{4\ell}\right)^{4\ell}}\right\}\nonumber\\
		B_{\ell}(s,y)&=\pi2^{\frac{5}{2}-2s-2\ell s}\sqrt{\ell}y^{2\ell s+s-1}\Gamma(s)\frac{\Gamma\left(\frac{1-s-2\ell s}{2}\right)}{\Gamma\left(\frac{2\ell s+s}{2}\right)}\pFq1{2\ell}{s}{\left\langle1+\frac{s(2\ell+1)-i}{2\ell}\right\rangle_{i=2\ell}}{(-1)^{\ell+1}\left(\frac{y}{2\ell}\right)^{2\ell}}\nonumber\\
		&=\frac{\pi2^{\frac{5}{2}-2s-2\ell s}\sqrt{\ell}y^{2\ell s+s-1}\Gamma\left(\frac{1-s-2\ell s}{2}\right)}{\Gamma\left(\frac{2\ell s+s}{2}\right)\Gamma(1-(2\ell+1)s)\sin\left(\frac{\pi s(2\ell+1)}{2}\right)}\nonumber\\
		&\quad\times\sum_{j=0}^{\infty}\frac{(-y)^{2\ell j}}{j!}\Gamma(j+s)\Gamma(1+j-(2\ell+1)(s+j))\sin\left(\frac{\pi}{2}((2\ell+1) s+2\ell j)\right)\nonumber\\
		&=\sqrt{\pi(k-1)}2^{2-s}y^{ks-1}\sum_{j=0}^{\infty}\frac{(-y)^{(k-1)j}}{j!}\Gamma(j+s)\Gamma(1+j-k(s+j))\sin\left(\frac{\pi}{2}(ks+(k-1)j)\right),
	\end{align}
	where the penultimate step follows easily by writing the series definition of the ${}_1F_{2\ell}$ and simplifying the gamma factors using \eqref{gmf}, and where, in the last step, we replaced $2\ell+1$ by $k$.
	
	Next, consider $A_{\ell}(s, y)$ defined in \eqref{meijer G simplification 2}. Denote by $C_{\ell}(s, y, j)$ the sum of the summands of $A_{\ell}(s, y)$ corresponding to $h=j$ and $h=j+\ell$, where $1\leq j\leq \ell$ so that 
	\begin{equation}\label{AC}
		A_{\ell}( s,y)=\frac{\sqrt{\pi}2^{\frac{3}{2}-s}}{\sqrt{\ell}}\sum_{j=1}^{\ell}C_{\ell}(s,y, j).
	\end{equation}
	It is easy to see that
	\begin{align*}
		C_{\ell}(s, y, j)=(-1)^{j-1}\frac{\Gamma\left(s+\frac{s-2j+1}{2\ell}\right)\Gamma\left(\frac{-1-s+2j}{2\ell}\right)}{\Gamma(2j-1)}y^{2j-2}\left\{\pFq2{4\ell+1}{\frac{j-1}{2\ell}+\frac{1-s}{4\ell},\frac{1}{2}+\frac{j-1}{2\ell}+\frac{1-s}{4\ell}}{\left\langle1+\frac{j-1}{2\ell}-b_{i}\right\rangle_{i=1\atop i\neq j}^{4\ell+2}}{\left(\frac{y}{4\ell}\right)^{4\ell}}\right.\nonumber\\
		\left.+\frac{(-1)^{\ell}y^{2\ell}}{2^{2\ell}}\frac{\left(\frac{1-s}{2\ell}+\frac{j-1}{\ell}\right)}{\left(s+\frac{s-1}{2\ell}-\frac{j-1}{\ell}-1\right)\left(j-\frac{1}{2}\right)_\ell(j)_\ell}\pFq2{4\ell+1}{\frac{j-1}{2\ell}+\frac{1-s}{4\ell}+\frac{1}{2},1+\frac{j-1}{2\ell}+\frac{1-s}{4\ell}}{\left\langle1+\frac{j-1}{2\ell}-b_{i}+\frac{1}{2}\right\rangle_{i=1\atop i\neq j+\ell}^{4\ell+2}}{\left(\frac{y}{4\ell}\right)^{4\ell}}\right\}.
	\end{align*}
	We now simplify the above expression using \eqref{hypergeometric identity} again. Note that in the first ${}_2F_{4\ell+1}$, one of the bottom parameters is $1+\frac{j-1}{2\ell}-\frac{(j+\ell)-1}{2\ell}=\frac{1}{2}$, whereas in the second ${}_2F_{4\ell+1}$, we have $ 1+\frac{(j+\ell)-1}{2\ell}-\frac{j-1}{2\ell}=\frac{3}{2}$ as one of its bottom parameters.
	
	For $1\leq i\leq 2\ell, i\neq j$, where $1\leq j\leq\ell$, define
	\begin{align*}
		\overline{b}_i&:=1+\frac{j-1}{2\ell}-b_i=1+\frac{j-i}{2\ell},\nonumber\\
		\overline{b}_{j+\ell}&:=	1+\frac{j-1}{2\ell}-\frac{j+\ell-1}{2\ell}=\frac{1}{2},\nonumber\\
		\overline{b}_{2\ell+1}&:=1+\frac{j-1}{2\ell}-\frac{s(2\ell+1)-1}{4\ell},\hspace{4mm}\overline{b}_{2\ell+2}:=\frac{1}{2}+\frac{j-1}{2\ell}-\frac{s(2\ell+1)-1}{4\ell},\nonumber\\
		\overline{b}_{2\ell+2+i}&:=1+\frac{j-1}{2\ell}-\frac{2i-1}{4\ell}=1+\frac{j-1}{2\ell}-\frac{1}{4\ell}.
	\end{align*}
	Thus, for $1\leq j\leq\ell$,
	\begin{align*}
		\overline{b}_{\ell+i}&=\overline{b}_i-\frac{1}{2},\hspace{5mm}
		\overline{b}_{2\ell+2+\ell+i}=\overline{b}_{2\ell+2+i}-\frac{1}{2},\\
		\overline{b}_{2\ell+2}&=\overline{b}_{2\ell+1}-\frac{1}{2},\hspace{5mm} \overline{b}_{j+\ell}=\frac{1}{2}.
	\end{align*}
	Moreover, if we let $z=(-1)^{\ell+1}\left(\frac{y}{2\ell}\right)^{2\ell}, a=\frac{j-1}{\ell}+\frac{1-s}{2\ell}$,  $c_i=2\overline{b}_{\ell+i}$ for $1\leq i\leq\ell$ and  $c_i=2\overline{b}_{2\ell+2+i}$ for $\ell+1\leq i\leq 2\ell$ in \eqref{hypergeometric identity} and note that
	\begin{align*}
		\frac{az}{c_1\cdots c_{2\ell}}=\frac{(-1)^{\ell}y^{2\ell}}{2^{2\ell}}\frac{\left(\frac{1-s}{2\ell}+\frac{j-1}{\ell}\right)}{\left(s+\frac{s-1}{2\ell}-\frac{j-1}{\ell}-1\right)\left(j-\frac{1}{2}\right)_\ell(j)_\ell},
	\end{align*}
	we find that
	\begin{align*}
		C_{\ell}(s, y, j)&=(-1)^{j-1}\frac{\Gamma\left(s+\frac{s-2j+1}{2\ell}\right)\Gamma\left(\frac{-1-s+2j}{2\ell}\right)}{\Gamma(2j-1)}y^{2j-2}\nonumber\\
		&\quad\times\pFq1{2\ell}{\frac{j-1}{\ell}+\frac{1-s}{2\ell}}{\left\langle1+\frac{j-i}{\ell}\right\rangle_{i=1\atop i\neq j}^{\ell},1-s+\frac{j-1}{\ell}-\frac{s-1}{2\ell},\left\langle1+\frac{j-i}{\ell}-\frac{1}{2\ell}\right\rangle_{i=1}^{\ell}}{(-1)^{\ell+1}\left(\frac{y}{2\ell}\right)^{2\ell}}.
	\end{align*}
	Finally, we show that
	\begin{align}\label{finally}
		\sum_{j=1}^{\ell}C_{\ell}(s, y, j)=\sum_{m=0}^{\infty}\frac{(-1)^m}{(2m)!}y^{2m}\Gamma\left(\frac{1+2m-s}{2\ell}\right)\Gamma\left(\frac{(2\ell+1)s-2m-1}{2\ell}\right).
	\end{align}
	This is true upon splitting the index $m$ of the right-hand side series into residue classes modulo $\ell$ and proving that the resulting sum is nothing but the left-hand side. Indeed, 
	\begin{align}\label{finally-2}
		&\sum_{m=0}^{\infty}\frac{(-1)^m}{(2m)!}y^{2m}\Gamma\left(\frac{1+2m-s}{k-1}\right)\Gamma\left(\frac{ks-2m-1}{k-1}\right)\nonumber\\
		&=\sum_{j=1}^{\ell}\sum_{m=0}^{\infty}\frac{(-1)^{\ell m+j-1}}{(2(\ell m+j-1))!}y^{2(\ell m+j-1)}\Gamma\left(\frac{1+2(\ell m+j-1)-s}{2\ell}\right)\Gamma\left(\frac{ks-2(\ell m+j-1)-1}{2\ell}\right)\nonumber\\
		&=\sum_{j=1}^{\ell}(-1)^{j-1}\frac{\Gamma\left(s+\frac{s-2j+1}{2\ell}\right)\Gamma\left(\frac{-1-s+2j}{2\ell}\right)}{\Gamma(2j-1)}y^{2j-2}D_{\ell}(s, y, j),
	\end{align}
	where
	\begin{align}\label{finally-1}
		D_{\ell}(s, y, j):=\sum_{m=0}^{\infty}\frac{(-1)^{\ell m}y^{2\ell m}\left(\frac{2j-1-s}{2\ell}\right)_m}{(2j-1)_{2\ell m}\left(s+\frac{s-2 j+1-2\ell}{2\ell}-m+1\right)_m}.
	\end{align}
	Employing the elementary identity 
	\begin{align*}
		\left(s+\frac{s-2 j+1-2\ell}{2\ell}-m+1\right)_m=(-1)^m\left(1-s+\frac{j-1}{\ell}-\frac{s-1}{2\ell}\right)_m,
	\end{align*}
	we see from \eqref{finally-1} and \eqref{finally-2} that our proof of \eqref{finally} will be complete provided it is shown that for $m\geq0$,
	\begin{align*}
		(2\ell)^{2\ell m}\prod_{i=1}^{\ell}\left(1+\frac{j-i}{\ell}\right)_m\left(1+\frac{j-i}{\ell}-\frac{1}{2\ell}\right)_m=(2j-1)_{2\ell m}
	\end{align*}
	Clearly, the identity holds for $m=0$. For $m\geq 1$, it follows easily by doing induction on $m$ and noting the elementary identity $(2\ell)^{2\ell}\prod_{i=1}^{\ell}\left(1+\frac{j-i}{\ell}\right)\left(1+\frac{j-i}{\ell}-\frac{1}{2\ell}\right)=(2j-1)_{2\ell}$, which itself follows straightforwardly from the definition of the shifted factorial.
	
	From \eqref{AC}, \eqref{finally}, and replacing $2\ell+1$ by $k$, we are led to
	\begin{align}\label{finally1}
		A_{\ell}(s, y)=\frac{\sqrt{\pi}2^{2-s}}{\sqrt{k-1}}\sum_{m=0}^{\infty}\frac{(-1)^m}{(2m)!}y^{2m}\Gamma\left(\frac{1+2m-s}{k-1}\right)\Gamma\left(\frac{ks-2m-1}{k-1}\right).
	\end{align}
	Finally, substitute \eqref{finally3} and \eqref{finally1} in \eqref{meijer G simplification 1.5}, then multiply both sides of the resulting identity by $\frac{2^{s-2}}{\sqrt{\pi(k-1)}\Gamma(s)}$ to arrive at \eqref{meijerg specialization eqn} upon simplification.
\end{proof}

\end{document}